\documentclass
[
11pt, 
]
{amsart}
\usepackage{
amssymb
}

\newtheorem{theorem}{Theorem}[section]
\newtheorem{proposition}{Proposition}[section]

\newtheorem{definition}{Definition}[section]
\newtheorem{corollary}{Corollary}[section]

\theoremstyle{remark} 
\newtheorem{example}{Example}[section]

\theoremstyle{remark}
\newtheorem{remark}{Remark}[section]

\numberwithin{equation}{section}

\date{\today}
\DeclareMathOperator{\supp}{supp}
\DeclareMathOperator{\Ker}{Ker}

\DeclareMathOperator{\WF}{WF}

\newcommand{\eps}{\varepsilon}
\newcommand{\R}{{\bf R}}
\newcommand{\Id}{\mbox{Id}}

\newcommand{\f}{g}

\renewcommand{\r}[1]{(\ref{#1})}
\newcommand{\PDO}{$\Psi$DO}
\newcommand{\be}[1]{\begin{equation}\label{#1}}
\newcommand{\ee}{\end{equation}}

\renewcommand{\d}{\mathrm{d}}
\renewcommand{\i}{\mathrm{i}}

\title[The identification problem]{The identification problem for the attenuated X-ray transform}

\author[Plamen Stefanov]{Plamen Stefanov}
\address{Department of Mathematics, Purdue University, West Lafayette, IN 47907}
\thanks{Partly supported by a NSF  Grant DMS-0800428}

\usepackage{graphicx}
\graphicspath{%
    {converted_graphics/}
    {/}
}
\begin{document}

\begin{abstract} 
We study the problem of recovery both the attenuation $a$ and the source $f$ in the attenuated X-ray transform in the plane. We study the linearization as well. It turns out that there are natural Hamiltonian flow that determines which singularities we can recover. If the perturbations $\delta a$, $\delta f$ are supported in a compact set that is non-trapping for that flow, then the problem is well posed. Otherwise, it may not be, and least  in the case of radial $a$, $f$, it is not. We present uniqueness and non-uniqueness results for both the linearized and the non-linear problem; as well as a H\"older stability estimate. 
\end{abstract}

\maketitle

\section{Introduction}
We study the attenuated X-ray transform 
\be{1}
X_af(x,\theta) = \int e^{-Ba(x+t\theta,\theta)} f(x+t\theta)\, \d t , \quad x\in \R^2, \;\theta\in S^1,
\ee
in the plane with a source $f$ and an attenuation $a$ that we want to recover. We denote by  
\be{2}
Ba(x,\theta) =\int_0^{\infty} a(x+t\theta)\,\d t
\ee
the ``beam transform'' of $a$, usually denoted by $Da$. 
We will assume that both $a$ and $f$ are compactly supported. In applications, a constant attenuation $a$ is also considered but when observations are made on the boundary of a compact domain, one can replace that constant by a constant multiple of the characteristic function of that domain. 

The problem that we study is: can we recover both $a$ and $f$ from knowledge of $X_af$? Sometimes this is called the \textit{Identification Problem} (for SPECT). 

This problem arises in Single Photon Emission Computerized Tomography (SPECT). Radioactive markers are injected into a patient's body and the emitted X-rays, attenuated by the body, are detected outside of it. The problem is to recover the source with a unknown attenuated coefficient. 

When $a$ is known, it is well known that $f$ can be reconstructed uniquely, even by means of explicit formulas \cite{Buk-Kaz, ArbuzivBK},  \cite{Novikov, Novikov_range}, \cite{Natterer_2001}. For this reason, some of the numerical attempts to do a reconstruction are focused on recovery, or getting a good approximation of $a$ first, instead of treating $(a,f)$ as a pair. Sometimes this is called \textit{attenuation correction}, see e.g., \cite{Welch1997,  Ramlau1998}.  In clinical applications, additional X-rays are taken to reconstruct $a$ first. Eliminating or reducing those additional X-rays remains an important problem. 

There has not been much progress in the mathematical understanding of the identification problem so far. A related but not identical problem for finding both a constant attenuation and the source in the exponential X-ray transform has been solved in  \cite{Solmon}, see also \cite{Hertle}. The main result in  \cite{Solmon} is, roughly speaking, that specific pairs of constant $a$ and radial $f$ cannot be distinguished but all other pairs can. The identification problem with $f$ a finite sum of  delta sources has been studied in   \cite{Natterer_81, Natterer_83}, see also \cite{Boman_89}, but the results there do not and cannot imply uniqueness. Natterer also viewed the problem as a range characterization problem: if the ranges of $X_{a_1}$ and $X_{a_2}$ happen to be the same, for example, then there cannot be uniqueness. Range conditions, see e.g., \cite{Novikov_range}, have been viewed as a possible tool for solving the problem, both numerically, see e.g., \cite{Bronnikov2000} and analytically, as in the recent work \cite{bal_SPECT}. Numerical reconstructions have been tried, too, with variable success, in \cite{Censor1979, Manglos1993, Welch1997,  Ramlau1998,Bronnikov, Bronnikov2000, HabibZaidi02012003}, for example. Some of them use clinical data. 
A.~L.~Bukhgeim \cite{Bukhgeim_Id} recently outlined a recovery algorithm   if $a$ is a priori known  to be a constant multiple of the characteristic function of a star-shaped domain.

Our approach is based on the following. 
 The attenuated X-ray transform, and its linearization, carry  information about  $f$ and $a$ along each line twice because we integrate both in $\theta$ and $-\theta$ directions. From a microlocal point of view, those two lines determine the wave front sets at covectors normal to them. So we have two equations for two unknowns. 
We study first a linear problem that appears as a linearization of $X_af$ near some fixed $(a_0,f_0)$. Also, the non-linear map  $X_{a_2}f_2-X_{a_1}f_1$  is of that form, see \r{nl.3}. This problem can be formulated as the inversion of $\mathcal{I}{\f} := I_{w_1}{\f}_1+I_{w_2}{\f}_2$, ${\f}=({\f}_1,{\f}_2)$, where $I_w$ is the weighted X-ray transform with a weight $w(x,\theta)$, see \r{2.1}. The weights $w_{1,2}$ are of specific type in the case of the Identification Problem but we study general weights first. The operator $\mathcal{I}$ is a Fourier Integral Operator but we do not study it directly. Instead, to analyze the equation $\mathcal{I}g=h$, we apply an explicit operator $Q$ to convert the equations 
\[
\left(I_{w_1}g_1 + I_{w_2}g_2\right)(z,\pm\theta) =h(z,\pm\theta) 
\]
to  equivalent  pseudo-differential ones of the type 
\be{0.1}
\left(w_1(x,\pm D^\perp/|D|)+\textnormal{l.o.t.} \right)g_1+ \left( w_1(x,\pm D^\perp/|D|) +\textnormal{l.o.t.} \right)g_1 =Qh,
\ee
 see Proposition~\ref{pr_eq}. 
Here, ``$\textnormal{l.o.t.}$'' stands for ``lower order terms'', and $w_j(x,\pm D^\perp/|D|)$ are  pseudo-differential operators (\PDO s) with  symbols $w_j(x,\pm \xi^\perp/|\xi|)$. We view this as a $2\times 2$ system of \PDO\ equations. The determinant of the principal symbol of the  is given by $p_0(x,\xi) = W(x,\xi^\perp/|\xi|)$, where 
\[
W(x,\theta) =  w_1(x,\theta)w_2(x,-\theta) -  w_1(x,-\theta)w_2(x,\theta)  .
\]
Since $W$ is an odd function of $\xi$, $p_0$  is not elliptic over any $x$, and has a non-trivial characteristic variety regardless of what $w_{1,2}$ are, in the cotangent bundle of any domain. Then $p_0$ is a Hamiltonian of fundamental importance for this system.  The singularities of ${\f}$ that may not be recoverable lie on zero bicharacteristics of that Hamiltonian; moreover each zero bicharacteristic either consists of singularities only, or there is none on it. This brings us to the following condition, well known in the theory of \PDO s of real principal type: if ${\f}$ has a support in a compact set $K$ that is non-trapping for the Hamiltonian flow, then the singularities of ${\f}$ can be recovered, with a loss of one derivative. Otherwise they may not be but the non-trapping condition is known not to be ``if and only if". In sections~\ref{sec_lin2} and \ref{sec_non_lin}, based on the microlocal understanding on the problem explained above, we  prove actual injectivity and stability of $\mathcal{I}$ for $f$ supported in non-trapping $K$ for generic $(w_1,w_2)$, including ones satisfying some analyticity assumptions; or for small $K$. We then apply this analysis to the non-linear Identification Problem to get local uniqueness and H\"older stability in a neighborhood of generic $(a,f)$ under the a priori assumption that the perturbations are supported in a non-trapping set.

The microlocal consequences of \r{0.1} are analyzed in more detail in Section~\ref{sec_micro}. In particular, we   describe the ``null eigenspace'' at the characteristic  points. In non-degenerate cases, \r{0.1} is of rank one on the characteristic variety $p_0=0$.

We also study the case of radial $a$ and $f$ in Section~\ref{sec_radial}. A thorough study of the radial case is behind the scope of this work however. The reason we include it is to present an example where the Hamiltonian flow can be explicitly computed.  
 The projections of the zero bicharacteristics happen to be the circles $|x|=R$, $R\ge0$. Then $K\subset \R^2$ is non-trapping if and only if it does not contain an entire circles centered at $0$, including the origin. In case $K$ is trapping, and contains a ball $|x|<R$, then the uniqueness fails and there is an infinite dimensional family of pairs $(a,f)$ with the same data. They consist of radial $a$ and $f$. This fact agrees with  the microlocal analysis that we present because the latter implies we may not be able to recover radial singularities. In this case actually, the non-trapping condition is also necessary for the problem to be well posed.

\section{Preliminaries} \label{sec_prel}
The attenuated X-ray transform  results from the following transport equation model. Let $f(x)$ be a compactly supported source of particles (or a signal propagating along lines with unit speed) propagating in a medium with attenuation coefficient $a(x)$. Then at the point $x\in \R^n$ and direction $\theta\in S^{n-1}$ (the dimension $n$ can be arbitrary), the total number of $u(x,\theta)$ of particles originating from the source solves the transport equation
\be{pr.1}
(\theta\cdot\partial_x+a)u=f, \quad u|_{\theta\cdot x\ll0}=0.
\ee
This is a linear ODE along the lines $t\mapsto (x+t\theta,\theta)$ and its solution is given by
\be{pr.2}
u(x,\theta) = \int_{-\infty}^0 e^{-\int_t^0 a(x+s\theta)\,\d s} f(x+t\theta)\,\d t.
\ee
This formula can be interpreted as the superposition of all attenuated signals at $(x,\theta)$ coming from the source.  
Then at points $x$ so that $\theta\cdot x\gg0$, one has $u=X_af$. 

It is useful to extend the definition of $B$, see \r{2}, to functions  $f$ depending on both $x$ and $\theta$:
\be{2a}
Bf(x,\theta) = \int_0^\infty f(x+t\theta,\theta)\,\d t.
\ee
For such $f$, the solution to \r{pr.1} is given by \r{pr.2} again, with $f(x+t\theta)$  replaced by $f(x+t\theta,\theta)$.

We introduce also the notation
\be{2.1}
I_wf(x,\theta) = \int w(x+t\theta,\theta) f(x+t\theta)\, \d t , 
\ee
for the weighted X-ray transform with weight $w(x,\theta)$. Then $I_w=X_a$ for $w=e^{-Ba}$ but we will allow more general weights in $I_w$. Also, $I_1=X_0$. 

We will also denote 
\[
v^\perp := (-v_2,v_1), \quad v=(v_1,v_2)\in \R^2.
\]

\subsection{A Radon transform type of parameterization of $X_a$ and $I_w$}
Since for a fixed direction $\theta$, $x$ and $x+s\theta$ parametrize the same (directed) line, we will think of $X_af$ and $I_wf$ as parameterized by $(z,\theta)$, $z\in \theta^\perp$. We denote by $Z$ the variety 
\[
Z = \{(z,\theta);\; \theta\in S^1, \, z\in \theta^\perp\},
\]
which is essentially the tangent bundle of $S^1$. Then we can set $z=p\theta^\perp$, and write $X_af$ as 
\be{2.1R}
X_af(p\theta^\perp,\theta) =   \int e^{-Ba(p\theta^\perp+t\theta, \theta)} f(p\theta^\perp+t\theta)\, \d t,\quad (p,\theta)\in \R\times S^1,
\ee
and similarly for $I_wf$. We think of $(p,\theta)\in \R\times S^1$ as a parameterization of $Z$.  We also define a measure on $Z$ by $\d z := \d p\,\d\theta$, where $\d\theta$ is the natural measure on $S^1$ given by $\d\vartheta$, with $\vartheta$ being the polar angle of $\theta$. 

\subsection{Functional spaces} We will assume throughout the paper that $\supp f$ is contained in a fixed compact set; and we can always assume that this compact set is  included in $(-\pi,\pi)^2$. We can therefore assume that $f$ is defined on the torus $\mathbf{T}^2$. For any compact set $K\subset \mathbf{T}^2$, we define $H^s(K)$ to be the closed subspace of $H^s(\mathbf{T}^2)$ of functions supported in $K$. In other words, the Sobolev norm in $K$ is defined through Fourier series. 
We define the Sobolev spaces $H^s(Z)$ in a similar way.  Since $|p|<\pi$ in \r{2.1R}, we can assume that $p$ belongs to the unit circle represented by $[-\pi,\pi]$ with both ends identified. Then $(p,\theta)\in S_p^1\times S_\theta^1$. The space $H^s(Z)$ is then defined by the norm
\be{A.4}
\|g\|_{H^s(Z)}= \big\| (1-\partial_p^2)^{s/2} g \big\|_{L^2(Z)},
\ee
where $\partial_p^2$ is the second  derivative w.r.t.\ $p$ on the compact manifold $S^1$. Notice that there are no $\theta$ derivatives in this definition, see also \cite[Theorem~II.5.2]{Natterer-book} for involving the $\theta$ derivatives when $a=0$. In other words, $H^s(Z)$ is defined through Fourier Series in the $p$ variable.

\section{Linearization}\label{sec_lin}
We are going to compute the linearization of the identification problem starting from formula \r{1}. Another way to do this, based on the transport equation,  is presented in section~\ref{sec_non_lin}. 

Assume that $a$ and $f$ are smooth enough so that the calculations below make sense. Denote by $G=\theta\cdot\partial_x$ be the generator of the geodesic flow on $T\R^2$ w.r.t.\ the Euclidean metric. 
Since  $a$ has compact support, then $GBa=-a$, and $Ba=0$ for $x\cdot\theta\gg0$;  and $Ba=I_1a$ for $x\cdot\theta\ll0$. Here, $I_1=I_a$ for $a=1$.

Since the problem is linear w.r.t.\ $f$, we linearize near some $a$ first, with $f$ fixed. 
Let $a_s = a+s{\delta a}$. Then 
\[
\frac{\d}{\d s}\Big|_{s=0}X_{a_s}f= -\int e^{-Ba(x+t\theta,\theta)}   f(x+t\theta)  B{\delta a} (x+t\theta,\theta)      \, \d t.
\]
Write 
\[
e^{-Ba}f= -GBe^{-Ba}f
\]
and plug this into the formula above. Integrate by parts to get
\[
\frac{\d}{\d s}\Big|_{s=0}X_{a_s}f= \int\left[ \left(B e^{-Ba}    f    \right) {\delta a}\right] (x+t\theta,\theta)  \, \d t - 
  X_af . X_0{\delta a}.
\]
The linearization of $X_af$ w.r.t.\ $a$ is therefore a weighted X-ray transform of the perturbation ${\delta a}$ of the form
\[
\int w(x+t\theta,\theta){\delta a}(x+t\theta) \, \d t 
\]
with a weight function 
\be{w}
w = B e^{-Ba}f - X_af.
\ee
The second term on the right is constant along each line. The weight can also be expressed as
\be{w2}
w(x,\theta) = -\int_{-\infty}^0 e^{-Ba(x+t\theta,\theta)}f(x+t\theta)\, \d t.
\ee
A direct calculation yields
\be{w2'}
w  = -e^{-Ba}u,
\ee
where $u$ is the solution \r{pr.2} of \r{pr.1}. 

Let $\delta X_{a,f}(\delta a,\delta f)$ denote the linearization of $X_af$ near fixed $a$, $f$. We just proved the following.

\begin{proposition}\label{pr1}
\[
\delta X_{a,f}(\delta a,\delta f) =  I_w\delta a + X_a\delta f ,
\]
where $w$ is as in \r{w2} or \r{w2'}. 
\end{proposition}

\section{A more general linear problem: invertibility of a sum of two weighted X-ray transforms}  \label{sec_lin2}
\subsection{Formulation and preliminaries}
Consider a more general problem. Let $\mathcal{I}({\f}_1,{\f}_2) = I_{w_1}{\f}_1+I_{w_2}{\f}_2$, where $w_{1,2}$ are two weight functions, i.e.,
\be{1.5}
\mathcal{I}({\f}_1,{\f}_2)(x,\theta) = \int w_1(x+t\theta,\theta){\f}_1(x+t\theta) \, \d t + \int w_2(x+t\theta,\theta){\f}_2(x+t\theta) \, \d t.
\ee
 We will compute $\mathcal{I}^*\mathcal{I}$ w.r.t.\ the inner product in $L^2(Z)$. Clearly,
\be{1.6}
\mathcal{I}^*\mathcal{I} = \begin{pmatrix}
I_{w_1}^* I_{w_1}&   I_{w_1}^* I_{w_2}   \\
          I_{w_2}^* I_{w_1}& I_{w_2}^* I_{w_2}
\end{pmatrix}
\ee
By Proposition~\ref{pr_symbol}, $\mathcal{I}^*\mathcal{I}$ is a \PDO\ of order $-1$ with  principal symbol
\be{1.7}
{\delta a}_p(\mathcal{I}^*\mathcal{I}) =  \frac{\pi}{|\xi|}\begin{pmatrix}
|w_{1,+}|^2 + |w_{1,-}|^2&   
\bar w_{1,+} w_{2,+} + \bar w_{1,-}w_{2,-}   \\
       w_{1,+}  \bar w_{2,+} + w_{1,-}\bar w_{2,-}  & 
|w_{2,+}|^2+ |w_{2,-}|^2
\end{pmatrix}
\ee
where
\[
w_{j,\pm} = w_j(x,\pm\xi^\perp/|\xi|), \quad j=1,2.
\]
A direct calculation yields
\be{det}
\det{\delta a}_p(\mathcal{I}^*\mathcal{I}) = |w_{1,+} w_{2,-} -w_{2,+}  w_{1,-}|^2 = \bigg|\det  \begin{pmatrix}
w_{1,+}& w_{2,+}   \\
     w_{1,-}     &   w_{2,-}
\end{pmatrix}\bigg|^2.
\ee
That determinant not being zero is a microlocal ellipticity condition. As we see below, it vanishes over any point $x$; therefore, $\mathcal{I}^*\mathcal{I}$ cannot be elliptic over (i.e., in the cotangent bundle of) any domain. 

Set
\be{W1}
W(x,\theta) =  w_1(x,\theta)w_2(x,-\theta) -  w_1(x,-\theta)w_2(x,\theta).
\ee
Then $\det{\delta a}_p(\mathcal{I}^*\mathcal{I})= |W(x,\xi^\perp)|^2$. The function $W$ is odd in $\theta$, and therefore, for any $x$ it has zeros for some vectors $\theta$. The inconvenience of working with \r{det} however is that it has double characteristics.

Instead of studying the invertibility of $\mathcal{I}^*\mathcal{I}$, we will approach the problem in a more direct way, slightly different (but equivalent) than what we do in Section~\ref{sec_micro}, see also \r{0.1}.  
Set 
\be{J}
Jh(x,\xi)= h(x,-\xi). 
\ee
Let $\alpha(x,\theta)$ be any smooth  function, odd on $S^1$ w.r.t.\ $\theta$. 
Apply $I'_{\alpha Jw_2}$ to the equation
\be{1.7eq}
I_{w_1}{\f}_1+I_{w_2}{\f}_2=h
\ee
to get
\be{1.8}
I'_{\alpha Jw_2}I_{w_1}{\f}_1+ I'_{\alpha Jw_2}I_{w_2}{\f}_2= I'_{\alpha Jw_2}h.
\ee
By Proposition~\ref{pr_symbol}, both operators on the left are \PDO s of order $-1$. The principal symbol of $I'_{\alpha Jw_2}I_{w_1}$ is given by $2\pi/|\xi|$ times the even part of $(\alpha w_1Jw_2)(x,\xi^\perp/|\xi|)$, i.e., by $2\pi\alpha(x,\xi^\perp/|\xi|)/|\xi|$ times the odd  part of $(w_1Jw_2)(x,\xi^\perp/|\xi|)$. Thus  
\be{det2}
\sigma_p( I'_{\alpha Jw_2}I_{w_1} ) = \frac{\pi}{|\xi|}\alpha W\big|_{\theta=\xi^\perp/|\xi|}.
\ee
Notice that $W$ is the determinant in the r.h.s.\ of \r{det} but not squared. It has the same zeros as \r{det} but they are simple. In the same way, we get that the principal symbol of $I'_{\alpha Jw_2}I_{w_2}$ is as above but with $w_2$ replaced by $w_1$, i.e., it is zero; and therefore, $I'_{\alpha Jw_2}I_{w_2}$  is of order $-2$.

Choose $\alpha=\theta_1$ first.  Then $|\xi|\alpha(\xi^\perp/|\xi|) =-\xi_2$, which is the symbol of $-D_2= \i\partial_2 $, and we get
\be{1.8a}
\sigma_p( I'_{\theta_1Jw_2}I_{w_1} ) = \frac{\pi}{|\xi|^2}(-\xi_2)W(x,\theta)\big|_{\theta=\xi^\perp/|\xi|}.
\ee
Modulo lower order terms, \r{1.8} becomes
\[
\frac{\pi}{|D|^2} (-D_2) W\left(x,D^\perp/|D|\right) {\f}_1 \cong I'_{\theta_1 Jw_2}h,
\]
where the meaning of $1/|D|$ is given by the \PDO\ calculus. Similarly, taking $\alpha=\theta_2$, we get
\[
\frac{\pi}{|D|^2}D_1 W\left(x,D^\perp/|D|\right) {\f}_1 \cong I'_{\theta_2 Jw_2}h.
\]
Apply $-D_2$ to the first identity,  $D_1$ to the second, and add them together to get
\[
\pi W\!\left(x,D^\perp/|D|\right) {\f}_1 \cong  \left(-D_2 I'_{\theta_1 Jw_2}+ D_1 I'_{\theta_2 Jw_2}\right) h.
\]
Notice that the lower order terms on the left involve ${\f}_2$ as well. 
In a similar way we get
\[
\pi W\!\left(x,D^\perp/|D|\right) {\f}_2 \cong   \left(D_2 I'_{\theta_1 Jw_1}- D_1 I'_{\theta_2 Jw_1}\right) h.
\]

We therefore proved the following. 
\begin{proposition}\label{pr_1.1}
For all compactly supported ${\f}_1$, ${\f}_2$   
we have
\be{1.12}
P{\f} = (\pi\i)^{-1}(\partial_1 I'_{\theta_2 Jw_2}-\partial_2 I'_{\theta_1 Jw_2},  -\partial_1 I'_{\theta_2 Jw_1}     +    \partial_2 I'_{\theta_1 Jw_1}  )\mathcal{I}{\f},
\ee
where ${\f}=({\f}_1,{\f}_2)$,  
and $P$ is a matrix valued classical \PDO\ of order $0$ with a scalar principal symbol given by 
\be{1.13}
p_0(x,\xi) :=  W\!\left(x,\xi^\perp/|\xi|\right).
\ee
\end{proposition}

We notice that \r{1.12} can also be written in the form
\[
\begin{split}
&\pi\i P{\f}\\ &= \left(\! -\int_{S^1}(\theta^\perp\!\cdot\partial_x) w_2(x,-\theta)\mathcal{I}{\f}(x-(\theta\cdot x)\theta,\theta)\,\d\theta, \int_{S^1}(\theta^\perp\!\cdot\partial_x) w_1(x,-\theta)\mathcal{I}{\f}(x-(\theta\cdot x)\theta,\theta)\,\d\theta
\right)\!.
\end{split}
\]

Let
\be{1.14}
\Sigma = \left\{(x,\xi)\in T^*\R^n\setminus 0;\; p_0(x,\xi)=0\right\}= \{W=0\}^\perp
\ee
be the characteristic variety of $p_0$, where the sign $\perp$ applies to the second variable $\theta$ only. 

There are several definitions of real principal type \PDO s in the literature, including or not the differential condition below, or the non-trapping one, in a fixed domain. We will use the following one.  We say that the \PDO\ $P\in \Psi^m$ is  of real principal type, if its principal symbol $p_m$ is real, scalar, homogeneous in $\xi$, and $\d p_m$ is not collinear to $\xi\d x$ on $\{p_m=0\}$ for $\xi\not=0$. The latter condition says that if we identify covectors of different length by their direction, then  the Hamiltonian vector field never vanishes, and in particular, the flow does not have stationary points. Such operators are microlocally equivalent to $\partial_{x_1}$ modulo lower order terms. We also note that this condition makes $\{W=0\}$ a codimension one (dimension $2$) smooth submanifold. The same applies to $\Sigma$, considered as part of  the unit cotangent bundle.

Below, $ \partial_{\theta^\perp}$ is the angular derivative in the $\theta$ variable, i.e., the derivative $\partial/\partial\vartheta$  w.r.t.\ to the polar angle $\vartheta$.

\begin{proposition}\label{pr_rpt}
The \PDO\ $P$ is of real principal type in some domain $\Omega\subset \R^2$, if an only if
\be{1.15}
\text{$W$, $\theta\cdot\partial_x W$, $\partial_{\theta^\perp} W$  
cannot be all zero at the same time, for any point in $\Omega\times S^1$}. 
\ee
\end{proposition}
\begin{proof}
Extend $W$ to $\theta\not=0$ as a homogeneous function of order $0$.   
We have
\[
\begin{split}
\d  p_0 &= \partial_x W(x,\xi^\perp)\, \d x + (\partial/\partial\xi) W(x,\xi^\perp)\, \d \xi\\
  &= W_x(x,\xi^\perp)\,\d x + (W_{\theta_2},-W_{\theta_1})(x,\xi^\perp)\,\d \xi.
\end{split}
\]
Then  $P$ is of real principal type if an only if $W(x,\xi^\perp)=0$ and $(\partial/\partial\xi) W(x,\xi^\perp)=0$ imply that $\partial_x W(x,\xi^\perp)$ is not collinear with $\xi$. The latter is equivalent to the requirement that  $\partial_x W(x,\xi^\perp)$ is not normal to $\xi^\perp$, i.e., $\xi^\perp\cdot \partial_x W(x,\xi^\perp)\not=0 $. Set $\theta=\xi^\perp$. Then the requirement is the following:  if $W(x,\theta)=0$ and $(\partial_{\theta^\perp} W)(x,\theta)=0$, then $\theta\cdot\partial_x W(x,\theta)\not=0$. Notice next that the radial derivative $\partial/\partial|\theta|$ of $W$ vanishes on $W=0$, therefore the requirement on $\partial_\theta W$ is actually a requirement on  the angular derivative only. 
\end{proof}

The following condition plays a critical role in the theory of local solvability of \PDO s of real principal type. Let $K$ be a compact subset of $\R^2$. 

\begin{definition} \label{def_nt}
We say that  $K$ is  \emph{non-trapping} (for $p_0$) if there is no maximally extended zero bicharacteristic that lies entirely over $K$. We call the projections of the zero bicharacteristics of $p_0$ to the $x$ space \emph{rays}. 
\end{definition}

Notice that the rays are continuous but not necessarily smooth curves. They may even degenerate to a point, see Example~\ref{ex_radial}. 

H\"ormander's  propagation of singularities theorem (see, e.g., \cite[VI.2.1]{Taylor-book0} or   \cite[Theorem~26.1.4]{Hormander4}) implies that if ${\f}$ is supported in a non-trapping $K$, and $P{\f}\in H^s$, then ${\f}\in H^{s-1}$. In other words, we have non-local  hypoellipticity, with a loss of one derivative. As a consequence, by the open mapping theorem, for any $s$ and $\ell$, there is $C>0$ so that
\be{1.18}
\|{\f}\|_{H^{s-1}}\le C\|P{\f}\|_{H^s} + C\|{\f}\|_{H^\ell},\quad \forall {\f}\in C_0^\infty(K),
\ee
see eqn.~(VI.3.3) in \cite{Taylor-book0}.

\begin{example}\label{ex_radial}  
Let $w_1=\frac12\theta\cdot x$ and $w_2=1$. Notice that $w_1$ is not non-vanishing. In this case, 
$W=\theta\cdot x$. Then $|\xi| p_0=x\cdot\xi^\perp = x_2 \xi_1-x_1\xi_2$ and   $\Sigma$ consists of $(x,\xi)$, $\xi\not=0$ that are collinear. In other words, all singularities that may not be recoverable are the radial ones. The Hamiltonian equations are given by
\[
\dot x_1 = x_2, \quad \dot x_2 = -x_1,\quad \dot \xi_1 = \xi_2, \quad \dot \xi_2 = -\xi_1.
\]
The zero bicharacteristics then are  given by
\be{1.18ch}
x= R(\sin t,\cos t), \quad \xi = \lambda (\sin t,\cos t), \quad R\ge0, \, \lambda\not=0. 
\ee
Their projections on the base (the rays) are given by the circles $x= R(\sin t,\cos t)$, $R\ge0$. If $R=0$, then that projection is a point. The whole bicharacteristic is not stationary however and is given by $x=0$, $\xi= \lambda (\sin t,\cos t)$, $\lambda\not=0$. We then see that a compact set K is non-trapping   if and only if $K$ contains no entire circle $|x|=R$, $R\ge0$ (including the origin), see Figure~\ref{fig:fig_identification2}.  
Then $\mathcal{I}{\f}$ recovers the singularities of ${\f}=({\f}_1,{\f}_2)$. If $K$ is trapping, the singularities that may not be possible to be recovered are the radial ones. 

Inverting $\mathcal{I}{\f} = I_{w_1}{\f}_1+I_{1}{\f}_2$ is easy. The first term is odd w.r.t.\ $\theta$, and the second one is even. The equation $\mathcal{I}{\f}=h$ then decouples into two equations $I_{w_1}{\f}_1=h_{\rm odd}$, $I_1{\f}_2=h_{\rm even}$. The kernel of $\mathcal{I}$ (on $\mathcal{E}'(\R^2)$) then consists of pairs $({\f}_1,0)$, where ${\f}_1\in \Ker I_{w_1}$. Using arguments similar to the Fourier Slice theorem, see Section~\ref{sec_micro},  we can see easily that $I_{w_1}{\f}_1=0$ if and only if $(x_2\partial_1-x_1\partial_2){\f}_1=0$ (the operator in the parentheses is $p_0(x,D)$, up to an elliptic  factor)   and the solutions of the latter in $\mathcal{E}'(\R^2)$ are given by all compactly supported radial distributions. In particular, on $L^2_{\text{comp}}(B(0,1))$, the kernel of $I_{w_1}$ consists of all radial  functions in that space.  We therefore get
\be{1.18Ker}
\Ker \mathcal{I} = \left\{ ({\f}_1,0);\; \text{${\f}_1$ is radial}\right\}.
\ee
Since radial functions can have (radial) singularities at all points, we get that no radial singularity of ${\f}_1$, i.e., a singularity of the type $(x=\mu \xi,\xi)$, $\mu\in \R$, $\xi\not=0$, can be recovered in general. On the other hand, if $K$ is non-trapping for $p_0$, and $\supp {\f}\subset K$, then they can. In that case, $I_{w_1}$ is microlocally  equivalent to the derivative w.r.t.\ to the polar angle in $\R^2$; and since on each circle there is an open arc where ${\f}=0$, that circle cannot support a singularity of ${\f}$, by the propagation of the singularities theorem. In fact, by \r{1.18Ker}, if $K$ is non-trapping, we have more: $\f_1=\f_2=0$.

This example also reveals that $\mathcal{I}{\f}\in C^\infty$ is not microlocally equivalent to \r{1.12}, see Section~\ref{sec_micro} for more details. Indeed, only the radial singularities of ${\f}_1$ are not recoverable, while those of ${\f}_2$ are recoverable. Moreover, 
assume that $\mathcal{I}{\f}\in H^s(Z)$. Then ${\f}_2\in H^{s-1/2}$ by the usual inversion results. We can think of $I_{w_1}{\f}_1$ as the Doppler X-ray transform of the vector field ${\f}_1(x)x$. It is well know that we can only reconstruct the curl of ${\f}_1(x)x$ that is $(x_1\partial_2 - x_2\partial_1){\f}_1$, and the latter is in $H^{s-3/2}$. Let $\supp {\f}_1\in K$, with $K$ non-trapping, for example, assume that the ray $\{x_1\ge0, \; x_2=0\}$ does not intersect $K$. Then, in polar coordinates $(r,\phi)$, 
\[
\f_1(r,\phi) = \int_0^\phi \left[  (x_1\partial_2 - x_2\partial_1)\f \right](r,\psi)\,\d\psi. 
\]  
This integration is not smoothing (not in all directions), and we still have ${\f}_1\in H^{s-3/2}$, with an improved regularity in angular directions. This is consistent with \r{1.18b} below but as we see, the one derivative loss is only in ${\f}_1$, and the that estimate does not reveal the extra regularity in characteristic directions. The latter is however reflected by the fact that $\WF({\f}_1)$ can have radial directions only. 
\end{example}

It is interesting to know when the rays are smooth curves. The projection of a bicharacteristic to the $x$ variables, with its parameterization determined by the Hamiltonian equation, at some $x$, has a tangent vector $\partial_\xi p_0$ evaluated at some $(x,\xi)\in \Sigma$ (i.e., $p_0(x,\xi)=0$). This projection is non-degenerate, and therefore, that ray is a smooth curve, if $\partial_\xi p_0\not=0$. 
There might be more than one $\xi$ with that property but there is at least one $\xi$ (and the whole line spanned by it) because $p_0$ is odd in $\xi$. 
On the other hand, $\xi\cdot\partial_\xi p_0=0$ on $\Sigma$, therefore a tangent vector is actually $\xi^\perp$ and the whole line that it spans. 

Translating this in terms of $W$, see \r{1.13}, we get the following. If for some $(x_0,\theta_0)$, we have
\be{1.18x}
\text{$W(x_0,\theta_0)=0$ and $\partial_{\theta^\perp} W(x_0,\theta_0)\not=0$},
\ee
then there is a smooth ray through it. Moreover,  starting from $(x_0,\theta_0)$ with that property, by the implicit function theorem, we can solve $W(x,\theta)=0$ locally for $\theta$. This gives us a smooth unit vector field, with integral curves that are rays.  
Then $W=0 \Rightarrow\partial_{\theta^\perp}W\not=0$ on $K\times S^1$ is a sufficient condition for all rays through $K$ to be smooth.

\subsection{Basic Properties of $\mathcal{I}$} 

Below, $\mathcal{E}'(K)$ stands for the space of distributions supported in $K$, and we similarly define $H_0^s(K)$. Also, since ${\f}$ is vector valued, $L^2(K)$, $H^S(K)$ below are spaces of vector valued functions. 

\begin{theorem}\label{thm_lin}
Let $\mathcal{I}$ be as in \r{1.5} with $w_1$ and $w_2$ smooth. Let $K\subset \R^2$ be a non-trapping compact set. Then for any $s\ge0$ we have the following.

(a) For any  $\ell$, there exist a constant $C>0$ so that
\be{1.18a}
\|{\f}\|_{H^s(K)}\le C \|\mathcal{I}{\f}\|_{H^{s+3/2}(Z)} + C\|{\f}\|_{H^\ell(\R^2)}, \quad \forall {\f}\in C_0^\infty(K). 
\ee

(b) The kernel of $\mathcal{I}$ on $\mathcal{E}'(K)$ is finite dimensional, and consist of $C_0^\infty(K)$ functions ${\f}$. 

(c) On the orthogonal complement of $\Ker \mathcal{I}$ in $H_0^s(K)$, we have
\be{1.18b}
\|{\f}\|_{H^s(K)}\le C \|\mathcal{I}{\f}\|_{H^{s+3/2}(Z)}. 
\ee
In particular, if $\mathcal{I}$ is injective on $C_0^\infty(K)$, then \r{1.18b} holds. 
\end{theorem} 

\begin{proof}
To prove (a), we will apply \r{1.18}. To this end, replace $P$ by $\chi P\chi$, where $\chi\in C_0^\infty$ equals $1$ near $K$ (this makes $P$ properly supported, in particular). Let $\Omega$ be a bounded domain containing $\supp\chi$.  Use estimate \r{1.18} combined with \r{1.12} to get for any fixed $\ell$, 
\be{1.20}
\begin{split}
\|{\f}\|_{H^s_0(K)} &\le C\sum_{i,j=1}^2\|I'_{\theta_iJw_j}\mathcal{I}{\f}\|_{H^{s+2}(\Omega)} + C\|{\f}\|_{H^\ell}\\
&\le C\| \mathcal{I}{\f}\|_{H^{s+3/2}(Z)}+C\|{\f}\|_{H^\ell}, \quad \forall {\f}\in C_0^\infty(K),
\end{split}
\ee
see also Proposition~\ref{pr_I}. Notice the one derivative loss in this estimate since $\mathcal{I}$ is of order $-1/2$, see the Appendix. If we replace $\mathcal{I}$ by a single weighted X-ray transform $I_w$ with a non-vanishing weight $w$, then one has the same estimate but with $H^{1/2}(Z)$. 
We also note that $\mathcal{I}{\f}$ has compact support. 

Consider (b). Every ${\f}\in \mathcal{E}'(K)$ in the kernel of $\mathcal{I}$ must be smooth by propagation of singularities and by the assumption that $K$ is non-trapping. Apply then \r{1.18a} to get 
\[
\|{\f}\|_{H^1(K)}\le  C\|{\f}\|_{L^2(K)}, \quad \forall {\f}\in \Ker \mathcal{I}\cap \mathcal{E}'(K) = \Ker \mathcal{I}\cap C_0^\infty(K). 
\]
Since the inclusion $H^1_0(K)\hookrightarrow  L^2(K)$ is compact, we get the finite dimensionality of $\Ker \mathcal{I}$ on $K$. 

Consider (c). Let $\mathcal{D}$ be the closure of $C_0^\infty(K)$ under the graph norm $\|{\f}\|_{H^s_0(K)}+\|\mathcal{I}{\f}\|_{H^{s+3/2}(Z)}$. We consider now $\mathcal{I}$ as an operator from $\mathcal{D}$ to $H^{s+3/2}(Z)$. Then $\mathcal{I}$ is a well defined bounded operator. Indeed, $\mathcal{D}$ is a subspace of the space of the compactly distributions, together with the topology. Then $\mathcal{I}$ can be considered as an operator originally defined as $\mathcal{I}: \mathcal{E}'(\R^2) \to \mathcal{E}'(Z)$, and then restricted to $\mathcal{D}$. We then get
\[
\|{\f}\|_{\mathcal{D}} \le C\| \mathcal{I}{\f}\|_{H^{s+3/2}(Z)}+C\|{\f}\|_{H^\ell}, \quad \forall {\f}\in \mathcal{D}.
\]
By (a), $\mathcal{I}$ is injective on $\mathcal{D}\cap (\Ker \mathcal{I})^\perp$. Then by \cite[Proposition~V.3.1]{Taylor-book0}, for $\ell<s$,  we have the same inequality as above  on $\mathcal{D}\cap (\Ker \mathcal{I})^\perp$ but without the last term. We refer also to \cite[Lemma~3]{SU-Duke} as well for similar arguments, or to inequality (26.1.6) in \cite{Hormander4}.
\end{proof} 

\subsection{Conditions for injectivity of $\mathcal{I}$} 
\begin{corollary}\label{cor_nt}
Let $w_1$ and $w_2$ be smooth. Let $x_0\in \R^2$ be such that 
\be{1.21}
W(x_0,\theta)=0 \quad \Longrightarrow\quad  \partial_{\theta^\perp}W(x_0,\theta)\not=0,\quad \forall \theta\in S^1. 
\ee
Then if $0<\eps\le1$, $\mathcal{I}$ is injective on distributions supported in the ball $B(x_0,\eps)$, and in particular, \r{1.18b} holds for $K= B(x_0,\eps)$. 
\end{corollary}

\begin{proof}
Condition \r{1.21} guarantees that any ray through $x_0$ is smooth at $x_0$, and there are finite number of such rays. There is $\eps_0>0$ so that $B(x_0,\eps_0)$ is non-trapping. 

Assume the opposite. Then for any  $\eps=1/j$, $j\ge1$, there is a non-trivial $C_0^\infty$ function supported in $B(x_0,1/j)$ in the kernel of $\mathcal{I}$.  
Then  we  get an infinite number of  non-trivial  functions $\phi_j$ in the finitely dimensional space $V= \Ker \mathcal{I}\cap C_0^\infty(B(x_0,\eps_0))$, see Theorem~\ref{thm_lin}(b), with supports shrinking to the point $x_0$. This is a contradiction. Indeed, $-\Delta :V\to-\Delta V$ must be a bounded operator. On the other hand, $-\Delta$ is bounded below on $H^1_0(B(x_0,1/j))\cap H^2$ by its first eigenvalue $\mu_j$, that tends to infinity as $j\to\infty$. Therefore, $(-\Delta \phi_j,\phi_j)/\|\phi_j\|^2\to\infty$, which is a contradiction.
\end{proof}

\begin{theorem}\label{thm_w_an}
Let $w_1$, $w_2$ be analytic in $\Omega\times S^1$, where  $\Omega$ is an open set containing  a non-trapping compact set $K\subset\R^2$.  Then the operator $\mathcal{I}$, restricted to $\mathcal{E}'(K)$, is injective. 
\end{theorem}

\begin{proof} 
We use a result about propagation of analytic singularities, see \cite{Hanges81}, for analytic \PDO s with real principal symbols. The result in \cite{Hanges81} covers in fact a more general class of operators with complex-valued principal symbols that have real bicharacteristics and carries over to operators with matrix lower order terms. 

The operator $A$ is an analytic \PDO\ in $\Omega$ of order $0$. Indeed, to prove that, it is enough to prove that operators of the kind $I_a'I_b$, see \r{1.18}, are analytic \PDO s (\cite{Treves}) of order $-1$ when $a$, $b$ are analytic in $\Omega\times S^1$. The amplitude of such an operator is given by \r{A.1amp}, and it is clearly an analytic one , see also the proof of \cite[Proposition~1]{SU-JAMS}.

The propagation of singularities result in \cite{Hanges81} then implies that each zero bicharacteristic of $P$ in $K$ either consists of (analytic) singular points only, or does not intersect the analytic wave front set of ${\f}$. Since $K$ is non-trapping, we have the latter alternative. Therefore, the analytic wave front set of ${\f}$ is empty. Then ${\f}$ is analytic. Since ${\f}$ is of compact support, we get ${\f}=0$.  
\end{proof}

\subsection{Generic injectivity of $\mathcal{I}$}

Let $K$ be a non-trapping compact set. Then any small enough compact neighborhood $K'$ of $K$ is still non-trapping, see the proof of \cite[Theorem~26.1.7]{Hormander4}. Therefore, there exists an open $\Omega\supset K$ so that every compact subset of $\Omega$ is non-trapping. Then $P$ is of real principal type in $\Omega$, by the definition in \cite{Hormander4}. 

\begin{definition}
The set $\Omega$ is said to be \textit{pseudo-convex} w.r.t.\ $P$, if any compact subset is non-trapping, and for any compact set $K_1\subset\Omega$, there exists a compact set $K_2\subset \Omega$ so that  every bicharacteristic interval in $\Omega$ having endpoints over $K_1$, lies entirely over $K_2$. 
\end{definition}

In particular, if $K_1$ is convex w.r.t.\ the bicharacteritics (i.e., one can choose $K_2=K_1$), then $K_1$ is pseudo-convex.

Pseudo-convexity is a condition that guarantees existence if a global parametrix of $P$, see \cite{Hormander_actaII} and  \cite[Theorem~26.1.14]{Hormander4}. Under that condition, we  show below  that injectivity of $\mathcal{I}$ is preserved under small perturbations of the weights.

\begin{theorem}\label{thm_pert}
Let $K$ be non-trapping for $P$ and assume that there exits a pseudo-convex neighborhood $\Omega\supset K$  of $K$. Assume that $\mathcal{I}$ is injective on $\mathcal{E}'(K)$. Then there exist $k>0$ and $\eps>0$ so that for any $\tilde w_1$, $\tilde w_2$, $\eps$-close to $w_1$ and $w_2$ in $C^k(\bar \Omega)$, the corresponding operator $\tilde{ \mathcal{I}}$ is still injective, and the estimate \r{1.18b} holds with a constant $C$ independent of the particular choice of $\tilde w_1$, $\tilde w_2$.
\end{theorem}

\begin{proof}
By  \cite{Hormander_actaII}, see also  \cite[Theorem~26.1.14]{Hormander4}, under the assumptions of the theorem one can construct a parametrix $E$ so that
\[
EP=\Id +R,
\]
where $R$ has a smooth kernel. The parametrix $E$ is not unique, even modulo smoothing operators. Loosely speaking, it is unique modulo smoothing operators if we fix an orientation on each connected set of bicharacteristics through $K$. The operator $E$ has the mapping property $E:H^s_0\to H^{s-1}$. If we make $P$ of order $1$, then $P$ would be microlocally equivalent to $\partial/\partial{x_1}$; and then roughly speaking, $E$ is integration w.r.t.\ $x_1$ in that representation in the direction of the chosen orientation.

By \r{1.12}, we have
\be{1.22}
EQ\mathcal{I}=\Id+R,
\ee
where $Q$ is of order $1/2$. Notice that $E$ is of order $1$, and $\mathcal{I}$ is of order $-1/2$. While the composition $EQ\mathcal{I}$ a priori is of order $1$ just based on the individual terms, it is actually of order $0$ as \r{1.22} shows. 

The construction of the Fourier Integral Operator (FIO) $E$ is described in \cite{Hormander4}. In order to get $R$ above to be just of order $-1$, all microlocal constructions need to be done up to finite order only in order to satisfy finitely many symbol estimates, see, e.g., \cite[Theorem~18.1.11']{Hormander3} and  \cite{SU-IBSP}.   In each step, finitely many derivatives of the symbols are needed; therefore, finitely many derivatives of $w_1$ and $w_2$ are needed.  Therefore, for some $k$, $C^k\ni (w_1,w_2)\to R$ is continuous, where $R: H^s\to H^{s+1}$ for a fixed $s$. 

The arguments below follow the proof of \cite[Proposition~5.1]{SU-JAMS}. The idea is to correct the parametrix $EQ$ by a finite rank operator so that  the new $\Id+R$ would be injective. We should be able to do this because $\mathcal{I}$ is injective.

Restrict equation \r{1.22} to $K$. In this stage of the proof, we will indicate the dependence on $w := (w_1,w_1)$ by a subscript $w$. 
We can always assume that $R_w$ is self-adjoint because we can apply $\Id+R_w^*$ to both sides of \r{1.22}.    
The operator $\Id+R_w$ has at most a finite-dimensional kernel $V$ on $L^2(K)$. Since $\mathcal{I}_w$ is injective on $L^2(K)$, $\mathcal{I}_w:V\to \mathcal{I}_wV$ is an isomorphism; let $B_w$ be its inverse. Let also $\Pi_{ w}$ be the orthogonal projection to $\mathcal{I}_{ w} V$. 
For $\tilde w$ close to $w$ as in the theorem, set $B^\sharp_{\tilde w}:= E_{\tilde w}Q_{\tilde w}+ B_w\Pi_{w} $. Then
\be{1.22est}
B^\sharp_{\tilde w}\mathcal{I}_{\tilde w}\f = (\Id + R^\sharp_{\tilde w})\f,
\ee
where $R^\sharp_{\tilde w}:= R_{\tilde w} +B_w\Pi_{w} \mathcal{I}_{\tilde w}$ is compact. We claim that $\Id + R^\sharp_{\tilde w}$ is injective for $\tilde w=w$. Indeed, assume $(\Id + R^\sharp_{ w})\f=0$. Then $(\Id +  R_{w})\f +B_w\Pi_{w} \mathcal{I}_{w}\f=0$. The first term is in $V^\perp$; the second one --- in $V$, therefore they are both zero. Thus $\f\in V$, and $B_w\Pi_{w} \mathcal{I}_{w}\f=0$. By the definition of $B_w$ and $\Pi_{w}$, this implies $g=0$. Therefore, $\Id + R^\sharp_{w}$ is injective, and actually invertible in $L^2(K)$.  This property is preserved under small $C^k$ perturbations of $w$, $k\gg1$,  as discussed above, with a uniformly bounded norm. The statement of the theorem now follows directly from \r{1.22est}.
\end{proof}

Theorem~\ref{thm_pert} and Theorem~\ref{thm_w_an} imply the following generic uniqueness result.
\begin{corollary}
Let $K$ and $\Omega$ be as in Theorem~\ref{thm_pert}. 
For some $k\gg1$, there is an open dense set of pairs $(w_1,w_2)$ in $C^k(\Omega)$ so that the corresponding operator $\mathcal{I}$ is injective on $\mathcal{E}'(\Omega)$,  and satisfies the stability estimate \r{1.18b} with a locally uniform constant. 
\end{corollary}

\section{The non-linear Identification Problem} \label{sec_non_lin}
Let $(a,f)$ and $(\tilde a, \tilde f)$ be two attenuation-source pairs. We will denote functions and operators related to $(\tilde a, \tilde f)$ by placing a tilde over them.  
The difference $v :=\tilde  u- u$ of the solutions of \r{pr.1} solves
\be{nl.1}
(\theta\cdot\partial_x+\tilde a)v= \delta f -   u\delta a,\quad v|_{\theta\cdot x\ll1}=0,
\ee
where
\be{nl.2}
 \delta a :=\tilde a- a, \quad \delta f:=\tilde f-  f.
\ee
Therefore,
\be{nl.3}
\tilde X_a\tilde f - X_af = I_w\delta a+ X_{\tilde a}\delta f,
\ee
where $I_w$ is the weighted X-ray transform with weight
\be{nl.4}
w  = -e^{-B\tilde a} u.
\ee
We used here the obvious generalization of \r{pr.2} for sources $f$ dependent on $\theta$ as well, see the remark following \r{2a}. If we replace $\tilde a$ on the right with $a$, then we get the linearization formula of Proposition~\ref{pr1}, as we should. 

\subsection{A summary of the properties of the linearization $\delta X_{a,f}$} 
We are in the situation of the previous section with 
\be{nl.5}
\f_1= \delta a, \quad \f_2 = \delta f, \quad w_1 =-e^{-B\tilde a} u , \quad w_2 = e^{-B\tilde a}.
\ee
If $\delta a$, $\delta f$ are given by \r{nl.2}, then \r{1.7eq} is a non-linear equation, of course. If we treat them as independent (of $a$, $\tilde a$, $f$, $\tilde f$)  functions then we have the linear problem that we analyzed above. Then 
\be{nl.5a}
W= e^{-B\tilde a}e^{-BJ\tilde a} W_0, \quad  W_0 := ( u-J u), 
\ee
see \r{W1}. The characteristic variety $\Sigma$ in this case is given by
\be{nl.6}
\Sigma = \left\{ u(x,\xi^\perp/|\xi|) =  u(x,-\xi^\perp/|\xi|)       \right\}.
\ee 
The Hamiltonian $p_0$ is then given by \r{1.13}. Since an elliptic factor does not change the zero bicharacteritics, just their parameterization, the zero bicharacteristics are then given by the following Hamiltonian
\be{nl.7}
H(x,\xi) =  \left(  u(x,\theta) -  u(x,-\theta)\right)\big|_{\theta=\xi^\perp/|\xi| }.
\ee
Recall that $ u$ us the solution of \r{pr.1}.

\begin{figure}[h] 
  \centering
  \includegraphics[bb=51 361 473 600,width=3.07in,height=1.73in,keepaspectratio]{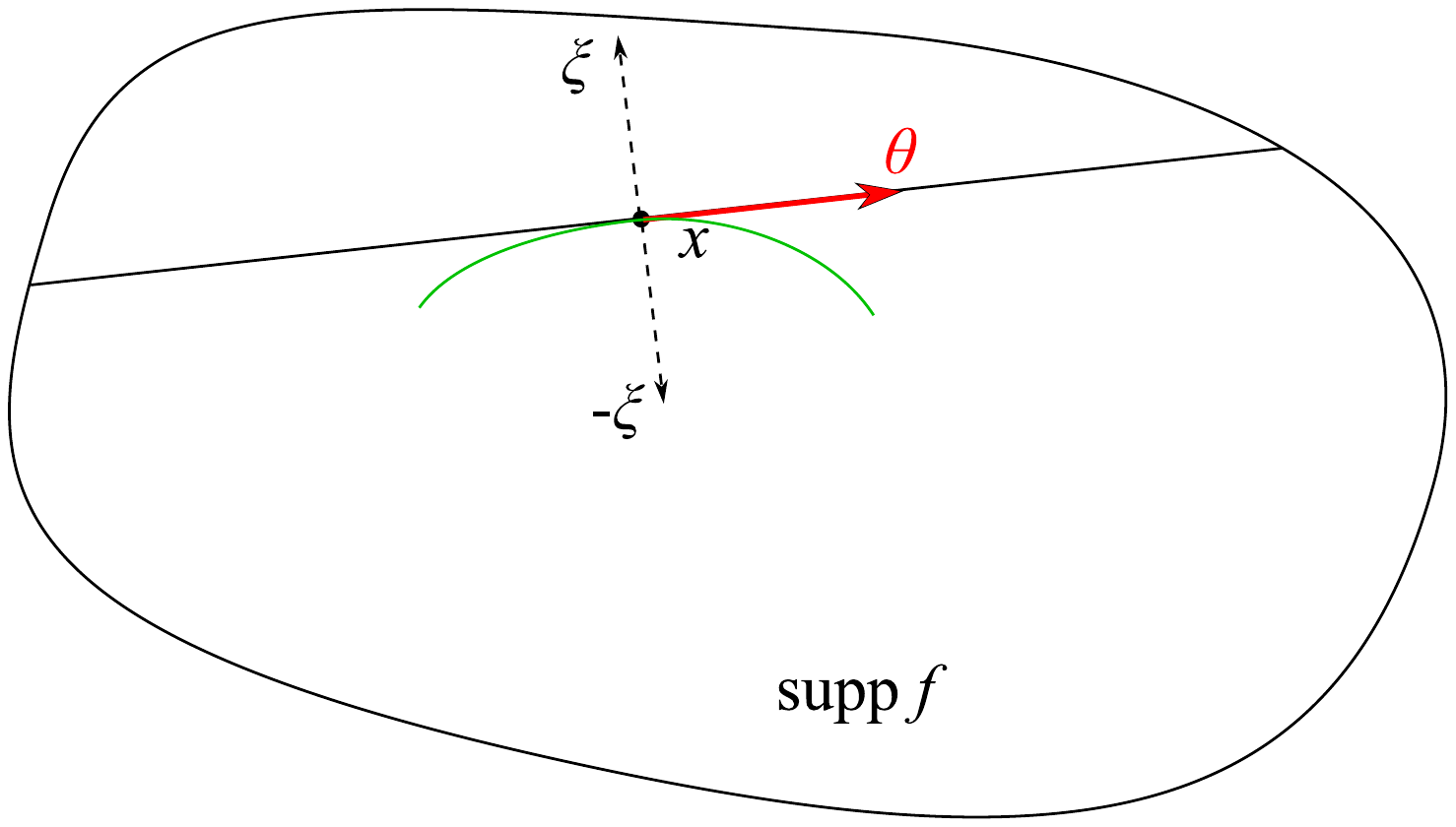}
  \caption{The zeros $(x,\theta)$ of $W$ are characterized by the property that the attenuated integrals of $f$ from $x$ in the directions $\theta$ and $-\theta$ are equal. The conormals $\xi$ to such $\theta$  are the characteristic ones. If $\partial_{\theta^\perp}W_0(x,\theta)\not=0$, then there is a smooth ray through  $x$ tangent to $\theta$.}
  \label{fig:fig_identification1}
\end{figure}

We will summarize the properties of the rays, see Definition~\ref{def_nt}, in this case. Let $\mathcal{I}$ be the linear operator defined in \r{1.5} with weights $w_1$ and $w_2$ as in \r{nl.5} but $\f_1 = \delta f$ and $\f_2 = \delta a$ considered as independent functions. 
Notice that the rays depend on $a$ and $f$ only. On the other hand, $\tilde a$, $\tilde f$ affect the weights in $\mathcal{I}$.
\begin{itemize}
  \item For any $x$ there is at least one ray through it which might be a point. 
  \item The rays may not be smooth. Given $(x,\theta)\in \R^n\times S^1$, there is a smooth ray trough $x$ in the direction of $\theta$ if and only if $W_0(x,\theta)=0$ and $\partial_{\theta^\perp}W_0(x,\theta)\not=0$. 
  \item Since $\theta\cdot\partial_xW_0=2f$, the condition $f(x)\not=0$ is  sufficient for $P$ to be of real principal type at $(x,\xi)$.
  \item A compact set $K\subset \R^2$ is called non-trapping, if all rays eventually leave $K$. 
  \item If $K$ is non-trapping, then $\mathcal{I}$, restricted to $K$, has a finite dimensional kernel, smooth enough if $B\tilde a$ and  $u$  are smooth enough near $K$. Also, \r{1.18a} holds. 
  \item If $\mathcal{I}$ is injective on $K$, then it is stable, as well, with a loss of one derivative, i.e., \r{1.18b} holds. If in addition $K$ has a pseudo-convex neighborhood, then the injectivity is preserved under a small enough perturbation with a uniform stability estimate  \r{1.18b}.
  \item If $W_0(x_0,\theta)=0$ implies $\partial_{\theta^\perp}W_0(x_0,\theta)\not=0$ for all $\theta$, then $\mathcal{I}$ is injective (and stable) restricted to functions supported in  some neighborhood of $x_0$. 
  \item  If $K$ is non-trapping, and $B\tilde a$ and  $u$ are analytic in a neighborhood of $K$, then $\mathcal{I}$ is injective (and stable).
\end{itemize}  

\begin{remark}\label{rem_support}
One important improvement in this case is to use the ellipticity of the second term $X_a$ in $\delta X_{a,f}$, see Proposition~\ref{pr1}. Let $\supp\delta a\subset K_1$, $\supp \delta f\subset K_2$, with $K_{1,2}$ compact sets. Then $\delta X_{a,f}$ is elliptic on the set $K_2\setminus K_1$ because there (i.e., for $\delta a$, $\delta f$ supported there), $\delta X_{a,f}(\delta a,\delta f) = X_a\delta f$. Therefore, for recovery of singularities we only need $K_1$ to be non-trapping. If, in addition,  $u\not=0$ on $\overline{K_1\setminus  K_2}$,  see \r{pr.1},  then it is enough to ask $K_1\cap K_2$ to be non-trapping. 
\end{remark}

\subsection{Uniqueness and stability results}
Our first main result about the identification problem is the following theorem. Recall that the requirement on $\Omega$ to be pseudo-convex implies that $K$ is non-trapping. 
\begin{theorem}\label{thm_local}
Let $K\subset\R^2$ be a compact set and let $\Omega\supset K$ be open. 
Let $a_0$, $f_0$ be  of compact support so that their beam transforms $Ba_0$ and $Bf_0$ are  in $C^k(\Omega\times S^1)$. Assume that $\Omega$ is pseudoconvex w.r.t.\ the   Hamiltonian $H$ defined in \r{nl.7}, related to $a_0$ and $f_0$. Let $a_0$, $f_0$ be such that $\delta X_{a_0,f_0}$, see Proposition~\ref{pr1}, is injective on $K$. Then if  $k\gg1$,    
there exists $\eps>0$ so that for any $(a_1,f_1)$, $(a_2,f_2)$ with $a_j-a_0\in C^k$ and $f_j-f_0\in C^k$ supported in $K$ satisfying 
\be{nl.7a}
\|B(a_j-a_0)\|_{C^k(\bar \Omega\times S^1)} + \|B(f_j-f_0)\|_{C^k(\bar \Omega\times S^1)}\le \eps,\quad j=1,2,
\ee
there exist constants $C>0$, $\mu\in (0,1)$ so that
\be{stab}
\|a_1-a_2\|_{L^2(K)}+\|f_1-f_2\|_{L^2(K)} \le C\|X_{a_1}f_1- X_{a_1}f_1\|_{H^{1/2}(Z)}^\mu.
\ee
\end{theorem}

\begin{proof} 
By \r{nl.3}, we have
\be{nl.8}
\begin{split}
X_{a_2}f_2- X_{a_1}f_1 &= -I_{e^{-Ba_1}u_1}\delta a + X_{a_1}\delta f +R\\
 &= \delta X_{a_1,f_1}(\delta a, \delta f)+R,
\end{split}
\ee
where 
\be{nl.7b}
R =  I_{(e^{-B a_1}-e^{-Ba_2})u_1} \delta a + X_{a_2-a_1}\delta f,
\ee
and $\delta a= a_2-a_1$, $\delta f = f_2-f_1$. 
Next, 
\be{nl.9}
\begin{split}
\|R\|_{L^\infty}&\le C \|B(a_2-a_1)\|_{L^\infty(K)} \|a_2-a_1\|_{L^\infty(K)} + C\|B(a_2-a_1)\|_{L^\infty(K)} \|f_2-f_1\|_{L^\infty(K)}\\
  &\le C'\left(\|\delta a\|_{L^\infty(K)}^2 + \|\delta f\|_{L^\infty(K)}^2\right),
\end{split}
\ee
where $C'$ depends on an a priori bound of $\|f_1\|_{L^\infty(\Omega)}$ which can always be found depending on $a_0$, $f_0$, $\eps$; by \r{stab}. 

We will apply \cite[Theorem~2]{SU-JFA}. Set $\mathcal{A}(a,f)=X_{a}f$. Set also  $\mathcal{B}_1=L^\infty(K)\times L^\infty(K)$, $\mathcal{B}_2=L^\infty(Z)$. Then $ \mathcal{A} :\mathcal{B}_1\to\mathcal{B}_2$ is continuous. 
By \r{nl.8} and \r{nl.9}, $\mathcal{A}$ is differentiable at $(a_1,f_1)$  with a quadratic estimate of the remainder. 

By assumption, $Ba_0|_{\bar \Omega\times S^1}$ and $Bf_0|_{\bar \Omega\times S^1}$ are in $C^k$. The same holds for the solution $u_0$ of \r{pr.1} related to $a_0$, $f_0$. 
Since $a_j-a_0\in C_0^k(K)$ and $f_j-f_0\in C^k_0(K)$, we also get the same for $Ba_j$, $Bf_j$, and $u_j$, $j=1,2$. Moreover, by \r{nl.7a},  $Ba_j$, $Bf_j$, and $u_j$, $j=1,2$ are $O(\eps)$ perturbations of  $Ba_0$, $Bf_0$, and $u_0$ in $C^k(\bar \Omega\times S^1)$. For $k\gg1$, we apply Theorem~\ref{thm_pert} to conclude that $\delta X_{a_1,f_1}$ is still injective, satisfying a stability estimate \r{1.18b} with a constant $C$ independent of $a_1$, $f_1$. Take $s>1$ in \r{1.18b}, for example, $s=3/2$,  to get
\[
\|\delta a\|_{L^\infty(K)}+ \|\delta f\|_{L^\infty(K)} \le C\|\delta X_{a_1,f_1}(\delta a,\delta f)\|_{H^{3}(Z)}.
\]
Based on that, we set
\[
\mathcal{B}_1'= \mathcal{B}_1=L^\infty(K), \quad \mathcal{B}_2'=H^{3}(Z).
\]
Then we have the following interpolation  estimate 
\[
\|h\|_{\mathcal{B}_2'}\le C \|h\|_{L^2}^\frac12 \|h\|_{H^{6}}^\frac12\le 
C' \|h\|_{\mathcal{B}_2}^\frac12 \|h\|_{\mathcal{B}_2''}^\frac12,
\]
where $\mathcal{B}_2''= H^{6}(Z)$. We have all conditions met to apply \cite[Theorem~2]{SU-JFA}. We therefore get that of $k\gg6$ ($k$ needs to satisfy both $k\ge6$ and the requirements of Theorem~\ref{thm_pert}), the stability estimate \r{stab} holds with $\mu=1/2$. 
\end{proof}

\begin{remark}
It is enough to assume that $a_j$ and $f_j$, $j=0,1,2$, satisfy the regularity  assumptions, instead of $Ba_j$, $Bf_j$ but that would be more restrictive. 
\end{remark}

\begin{remark}
The support conditions for $a_j-a_0$ and $f_j-f_0$ can be relaxed to some extent as in Remark~\ref{rem_support}.
\end{remark}

\begin{remark}
The value for $\mu$ that we got is $\mu=1/2$ but that was based on specific, and a bit arbitrary choice of the interpolation space $H^6$.  
As shown in \cite[Theorem~2]{SU-JFA}, and as can  be easily seen from the proof, we can choose any $\mu>1$ in \r{stab}, as close to $1$ as we wish, at the expense of increasing $k$. 
\end{remark}

Next corollary is a generic local uniqueness and stability result on non-trapping sets.

\begin{corollary} \label{cor_non1}
Let $K\subset\R^2$ be a compact set and $\Omega\supset K$ be open. 
Let $a_0$, $f_0$ be  of compact support so that their beam transforms $Ba_0$ and $Bf_0$ are analytic in $\Omega\times S^1$. Assume that $\Omega$ is pseudoconvex w.r.t.\ the   Hamiltonian $H$ defined in \r{nl.7}, related to $a_0$ and $f_0$. Then the conclusions of Theorem~\ref{thm_local} hold. 
\end{corollary}

This corollary implies in a trivial way also local uniqueness, in $K$, near a generic (dense and open in $C^k$, $k\gg1$) set of $(a,f)$. The proof follows immediately from Theorem~\ref{thm_w_an}.

The second corollary below states local uniqueness and stability in a small enough non-trapping set. 
\begin{corollary} \label{orthm_small_d} 
Let $x_0\in R^2$ is such that $Ba$, $Bf$ are smooth near $x_0$, and $W_0$ satisfies \r{1.21}. Then there exists an open set $U\ni x_0$, so that for any $K\subset U$ the conclusions of Theorem~\ref{thm_local} hold. 
\end{corollary}

\begin{proof}
By Corollary~\ref{cor_nt}, if $U$ is small enough, $\delta X_{a,f}$ is injective on any compact set $K\subset U$. Then we apply Theorem~\ref{thm_local}.
\end{proof}

\subsection{Conditions for smoothness and analyticity of $Ba$ and $Bf$} 
The results above require $Ba$ and $Bf$ to be either smooth enough or analytic in some open set $\Omega$. The smoothness, for example,  certainly hold if $a$ and $f$ are smooth enough in $\Omega$ but this is too restrictive. The following condition is sufficient.

\begin{proposition}\label{pr_sm}
Let $\Omega\subset\R^2$ be open. 
Let $\{c_j\}_{j=1}^N$ be a finite number of $C^k$ (respectively analytic) non-intersecting curves in $\R^2\setminus\Omega$ so that $a$ and $f$ are $C^k$/analytic in $\R^2\setminus\{c_j\}$,  up to the boundary on either side of each curve. Assume that each line through $\Omega$ intersects every $c_j$ transversely. Then $Ba$, $Bf$ are in $C^k(\Omega\times S^1)$, respectively analytic in $\Omega\times S^1$. 
\end{proposition} 

\begin{proof}
Near each $(x_0,\theta_0)\in \Omega\times S^1$, $Ba$, and similarly $Bf$, is given by
\[
Ba(x,\theta) = \sum_{j=1}^N \int_{\alpha_j(x,\theta)}^{\alpha_{j+1}(x,\theta)} a(x+t\theta,\theta)\,\d t,
\]
where $\alpha_0=0$, $\alpha_{N+1}=\infty$, and the rest of the $\alpha_j$'s are determined by the intersection points of the ray $x+t\theta$ with the curves $c_j$. The statement now follows directly from this representation.
\end{proof}

\begin{remark}
We presented the condition above in a form suitable for applications. For $C^\infty$, respectively, analytic regularity of $a$, $f$ in $\Omega\times S^1$, it is necessary and sufficient to assume that $a$ and $f$ have the same regularity in $\Omega$; and $a$, $f$, have no $C^\infty$, respectively analytic singularities, conormal to some line through $\Omega$. The necessity follows from standard properties of the Radon transform to recover conormal smooth or analytic singularities. This condition is sufficient, because of the standard relation between the smooth/analytic wave front set of $Ba$ or $Bf$ on one side; and the Schwartz kernel of $B$ and $a$ or $f$, on the other. We sill skip the details.
\end{remark}

\section{Further microlocal properties of $\mathcal{I}$} \label{sec_micro} 
Take the Fourier transform of $I_w(p\theta^\perp,\theta)$  w.r.t.\ $p$ to get
\be{nl.12aa}
\int e^{-\i \lambda p}  I_wf(p\theta^\perp,\pm\theta)\, \d p = \int e^{-\i \lambda \theta^\perp\cdot y }w(y,\pm \theta)f(y)\,\d y.
\ee
Set $\xi=\lambda\theta^\perp$, $\lambda\ge0$, to get
\be{nl.12a}
\int_{\R} e^{-\i p|\xi|}(I_wf)(p\theta^\perp , \pm \xi^\perp/|\xi|)\,\d p= \int e^{-\i y\cdot\xi} w(y, \pm \xi^\perp/|\xi|)f(y)\,\d y.
\ee
Take the inverse Fourier transform of both sides to get
\be{nl.13}
\bar w^*(x,\pm D^\perp/|D|)f= (2\pi)^{-2} \int_{\R\times \R^2} e^{\i( x\cdot\xi -p|\xi|)}(I_wf)(p\xi/|\xi|, \pm \xi^\perp/|\xi|) \,\d p\,\d \xi,
\ee
where $\bar w^*(x,\pm D^\perp/|D|)$ is the \PDO\ with amplitude $\alpha(x,y,\xi) = w(y,\pm \xi^\perp/|\xi|)$. The principal symbol is $w(x,\pm \xi^\perp/|\xi|)$. 
If $w=1$, one can see that we get $C|D|I_1'I_1f$ on the right, and $f$ on the left, which is just one of the inversion formulas for $I_1$.

Apply the described operation to the equation
\be{nl.11}
I_{w_1}{\f}_1+I_{w_2}{\f}_2=0,
\ee
compare with \r{1.7eq}. We get
\be{nl.14}
\bar w_1^*(x,\pm D^\perp/|D|){\f}_1 + \bar w_2^*(x,\pm D^\perp/|D|){\f}_2=0.
\ee
This is actually a system, see also \r{0.1}. 

\begin{proposition}\label{pr_eq}
Let $w_{1,2}$ be two smooth weight functions, and let ${\f}=({\f}_1,{\f}_2) \in \mathcal{E}'(\R^n)$. Then 
\be{nl.14a}
\mathcal{I}{\f}\in H^s(Z)
\ee
if and only if
\be{nl.15}
\begin{pmatrix}
\bar w_{1}^*(x,D^\perp/|D|)  &\bar w_{1}^*(x,D^\perp/|D|)   \\
   \bar  w_{2}^* (x,-D^\perp/|D|)    & \bar  w_{2}^*(x,-D^\perp/|D|)
\end{pmatrix}\! {\f}\in H^{s-1/2}(\R^2). 
\ee
\end{proposition}
\begin{proof}
Assume that the l.h.s.\ of \r{nl.15} is in $H^{s-1/2}$. Then the r.h.s.\ of \r{nl.13} with $I_w{\f}$ replaced by $\mathcal{I}{\f} := I_{w_1}{\f}_1+I_{w_2}{\f}_2$ belongs to the same space. Take the Fourier transform of that to get, see also \r{nl.12aa}, 
\be{nl.16}
\langle \lambda\rangle^{s-1/2}\int_{\R} e^{-\i \lambda p}\mathcal{I}{\f}(p\theta^\perp, \pm \theta)\,\d p \quad \in\quad L^2\left(\R_+\times S^1,\, \lambda\,\d\lambda\,\d\theta\right).
\ee
Since the relation above holds with either choice of the $\pm$ sign, we can fix the positive one, and allow $\lambda$ to be negative, as well. Therefore, $\langle\lambda\rangle^{s-1/2}|\lambda|^{1/2} \mathcal{F}_{p\mapsto\lambda} \mathcal{I}(p\theta^\perp,\theta) \in L^2\left(\R\times S^1,\, \d\lambda\,\d\theta\right)$. 
 This easily implies, see e.g., the proof of  \cite[Theorem~II.5.1]{Natterer-book}, that  $ \langle \lambda \rangle^{s}\mathcal{F}_{p\mapsto\lambda} \mathcal{I}(p\theta^\perp,\theta) \in L^2(\R\times S^1)$, which yields \r{nl.14a}. 

Now, assume \r{nl.14a}. Reversing the arguments above, we get \r{nl.16}. Take inverse Fourier transform w.r.t.\ $\xi=\lambda\theta^\perp$ to get \r{nl.15}.
\end{proof}

Proposition~\ref{pr_eq} reduces the problem of the microlocal invertibility of the FIO $\mathcal{I}$ to that of the matrix valued \PDO\ in \r{nl.15} with a principal symbol 
\be{nl.17}
\begin{pmatrix}
 w_{1}(x,\theta)  &w_{2}(x,\theta)   \\
   w_{1} (x,-\theta)    &  w_{2}(x,-\theta)
\end{pmatrix}\bigg|_{\theta=\xi^\perp/|\xi|}. 
\ee
The determinant of the latter is $W(x,\xi^\perp/|\xi|)$, see \r{det} and \r{W1}. An immediate consequence of \r{nl.15} is the following. For some matrix valued classical \PDO\ $\tilde P$ with a scalar principal symbol $p_0(x,\xi) = W(x,\xi^\perp/|\xi|)$, see \r{1.13}, relation \r{nl.14a} implies
\be{nl.18}
\tilde P {\f} \in H^{s-1/2}(\R^2).
\ee
This also follows from Proposition~\ref{pr_1.1}.

Assume now  that \r{1.18x} is satisfied for some $(x_0,\theta_0)$. Then we can solve the equation $W(x,\theta)=0$ for $\theta\in S^1$ locally to get a smooth function $\theta(x)$. Since $W$ is an odd function of $\theta$, the same thing applies near the point $(x_0,-\theta_0)$, as well, with a solution $-\theta(x)$. This implies that in a conic neighborhood of $(x_0,\pm\theta_0^\perp)\in \Sigma$, the characteristic manifold $\Sigma$ is given by $\xi^\perp/|\xi^\perp| = \pm \theta(x)$. Set $v_j^\pm(x) = w_j(x,\pm \theta(x))$, $j=1,2$. Then $w_j(x,\pm \xi/|\xi^\perp)-  v_i^\pm(x)$ vanishes on $\Sigma$, and is therefore locally given by $p_0(x,\xi)$ times a smooth function, homogeneous of order $0$ in $\xi$, hence a symbol. This implies that \r{nl.15} can be written as 
\be{nl.17c}
\begin{pmatrix}
 v_1^+(x)  &v_2^+(x)   \\
   v_2^+(x)    &  v_2^-(x)
\end{pmatrix} {\f} + (Q_0\tilde P+Q_{-1}){\f}\in H^{s-1/2}(x_0,\xi_0),
\ee
where $Q_0$ and $Q_{-1}$ are classical \PDO s of order $0$ and $-1$, respectively. 
 Using \r{nl.18}, we get
\be{nl.18n}
v_1^\pm {\f}_1+ v_2^\pm {\f}_2+Q^\pm _{-1}{\f}\in H^{s-1/2}(x_0,\xi_0),
\ee
with $Q^\pm _{-1}$ of order $-1$,  
and the equations with the $+$ and the $-$ sign are actually linearly dependent up to the lower order term (including the possibility that one of them has zero coefficients). Now, if the assumptions of Theorem~\ref{thm_lin} are satisfied, \r{nl.14a} yields ${\f}\in H^{s-3/2}$. Then $Q^\pm _{-1}{\f}\in H^{s-1/2}$, and 
we get
\be{nl.18a}
v_1^\pm {\f}_1+ v_2^\pm {\f}_2 \in H^{s-1/2}(x_0,\xi_0).
\ee
Since the matrix in \r{nl.17c} has rank $1$, only one of the equations \r{nl.18a} is relevant. 
This is an improvement over the estimate \r{1.18b}, that asserts that \r{nl.14a} implies ${\f}_{1,2}\in H^{s-3/2}$, if $\supp {\f}$ is supported in a non-trapping compact set. This improvement applies to the linear combination \r{nl.18a} only. 

\subsection{Applications to the linearized Identification Problem} 
Let $\mathcal{I}= \delta X_{a,f}$ be the linearization of $X_af$, see Proposition~\ref{pr1}. Then $f_j$, $w_j$ are given by \r{nl.5}. The determinant $W$ can be replaced in the analysis by $W_0$, see \r{nl.5a}. Notice that $w_2>0$. The discussion above yields the following.

\begin{proposition}\label{pr_WF}
Fix $(x_0,\theta_0)\in \R^2\times S^1$. Let $W_0(x_0,\theta_0)=0$ and $\partial_{\theta^\perp}W_0(x_0,\theta_0)\not =0$. Let
\be{nl.v}
v(x) = u(x,\theta(x)),\quad \text{for $x$ near $x_0$},
\ee
where $\theta(x)$ is the unique local solution of $W_0(x,\theta)=0$ with $\theta(x_0,\theta_0)=\theta_0$, and $u$ is defined by  \r{pr.2}. Then, if the assumptions of Theorem~\ref{thm_lin} are satisfied, and if 
 $\delta X_{a,f} (\delta a, \delta f)\in H^s$, we  have 
\be{nl.19}
v\delta a - \delta f\in H^{s-1/2}(x_0,\pm \theta_0^\perp).
\ee
\end{proposition}
\begin{proof}
In this particular case, $w_1= -e^{-Ba}u$, $w_2= e^{-Ba}$. Under the non-degeneracy assumption on $W_0$, $w_2>0$, and $w_1=-uw_2$. Divide by the elliptic factor $w_2$ in either of the two relations \r{nl.18a} to get \r{nl.19}. 
\end{proof}

\begin{remark}
Theorem~\ref{thm_lin} says that under the non-trapping condition we can recover $WF(f_{1,2})$ with a loss of one derivative, compared to the standard X-ray transform. On the other hand, Proposition~\ref{pr_WF} says that under the additional mild condition on $W$, one can recover the wave front set of the linear combination \r{nl.19} without loss. This has the following implications for the recovery of $a$ and $f$: we can expect $v\delta a - \delta f$ to be recoverable in a more stable way than either $\delta a$ or $\delta f$. 
\end{remark}

\begin{remark}\label{rem4}
We need to assume that the assumptions of Theorem~\ref{thm_lin} are satisfied just to conclude that $f\in H^{s-3/2}$; and then to deduce that $Q^\pm _{-1}f\in H^{s-1/2}$, see \r{nl.18n} and \r{nl.18a}. If we know a priori that $f$ has certain regularity, then we can use that fact instead. In applications, it would be natural to assume that $(\delta a, \delta f)\in L^2$. Let us assume that the measurements show that $\delta X_{a,f} (\delta a, \delta f)\in H^{3/2}$ (or better). Then we conclude that $v\delta a-\delta f\in H^{1}$, that in particular excludes jump types of singularities at smooth surfaces of that particular linear combination. There is no need to assume the trapping condition for this argument. 
\end{remark}

\section{The radial case} \label{sec_radial}
As explained in the Introduction, the thorough study of the case of radial $a$ and $f$ is behind the scope of this work. The purpose of this section is to present a case, where the rays can be easily computed, when both the linearized map, and the non-linear one have huge kernels if the non-trapping assumption is not satisfied. So at least in the cases described below, the non-trapping assumptions is not only sufficient but also necessary for the problem to be ``well-behaved''. 

\subsection{The linearized map for a simple radially symmetric example}
We start with perhaps the simplest example. Let $\mathbf{1}_{B(0,1)}$ be the characteristic function of the unit disk. We study the linearization $\delta X$ w.r.t.\ $(a,f)$ near
\be{r.1}
a=0, \quad f=\mathbf{1}_{B(0,1)}.
\ee
We will choose perturbations of those $a$ and $f$ supported in $B(0,1)$ only. 
The weight $w$, see \r{w2} or \r{w2'}, restricted to the unit disk, is given by
\be{r.2}
w(x,\theta) = 
-\sqrt{1-(\theta^\perp\cdot x)^2}-\theta\cdot x.
\ee
Then, see \r{nl.5a}, 
\be{r1}
W_0 = -2\theta\cdot x.
\ee
The Hamiltonian $H$, up to a constant factor, is as in Example~\ref{ex_radial}. Indeed, by  \r{nl.7}, $H= -2x\cdot\xi^\perp/|\xi| = -2(x_1\xi_2-x_2\xi_1)/|\xi|$. Therefore, $-2|\xi|H$ is the symbol of 
\[
x_1D_2-x_2D_1 = -\i \partial/\partial\phi, 
\]
where $\phi$ is the polar angle in the $x$ space. The bicharacteristics are given by  \r{1.18ch}. In particular, the rays are the concentric circles $|x|=R$, $R\ge0$, including the degenerate case $x=0$. As before, $K\subset B(0,1)$ is non-trapping, if and only if $K$ does not contain an entire circle of that kind, see Figure~\ref{fig:fig_identification2}.

\begin{figure}[h]
  \centering
  \includegraphics[bb=0 0 480 197,width=4.42in,height=1.81in,keepaspectratio]{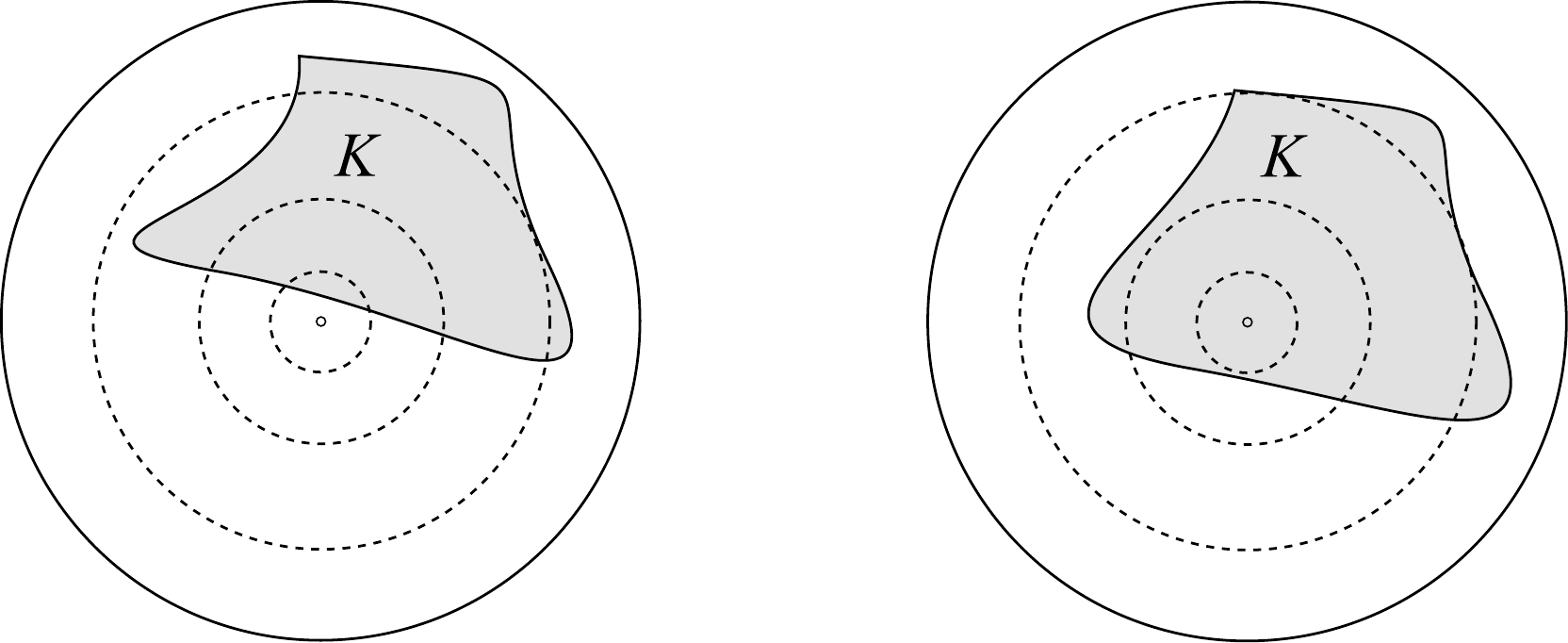}
  \caption{The rays  of Example~\ref{ex_radial} in the unit disk and an example of a non-trapping $K$, left; and a trapping $K$, right.}
  \label{fig:fig_identification2}
\end{figure}

The equation $\delta X(\delta a, \delta f)=0$ can then be written as
\[
-\int_{\ell_{z,\theta}} \left( \sqrt{1-(\theta^\perp\cdot x)^2}+\theta\cdot x \right) \delta a\, \d s+\int_{\ell_{z,\theta}}\delta f\,\d s=0,
\]
where $\ell_{z,\theta}$ is the line through $z\in\theta^\perp$ in the direction of $\theta$, and $\d s$ is the natural measure on it. The integral over the line $\ell_{z,-\theta}$ would produce the same term with $\theta\cdot x$ replaced by $-\theta\cdot x$. Therefore, both the even and the odd part w.r.t.\ $\theta$ above vanish:
\be{nl.10}
-\int_{\ell_{z,\theta}}  \sqrt{1-(\theta^\perp\cdot x)^2} \delta a\, \d s+\int_{\ell_{z,\theta}}\delta f\,\d s=0,\qquad \int_{\ell_{z,\theta}}\theta\cdot x   \delta a\, \d s=0.
\ee
The third integral is the X-ray transform of the vector field $(\delta a)x$. It is well known that we can only determine the curl of that, i.e.,
\[
(x_1\partial_2-x_2\partial_1)\delta a=0.
\]
In other words, $\delta a$ needs to be radial. Then the first  term in \r{nl.10} is invariant under rotations of $(x,\theta)$, i.e., when we consider $(x,\theta)$ as points in the unit tangent bundle. Then so is the second term. Apply $I_1'$ to it, and we get that $|D|^{-1}\delta f$ is radial, as well. Then so is $\delta f$. 

Therefore, the kernel of $\delta X(\delta a, \delta f)$ consists of radial $\delta a$ and $\delta f$ that are connected by the first identity in \r{nl.10}. Since the weight there is constant along the lines, using  Radon transform notation, $Rh(p,\omega)$, we get
\be{r5}
\sqrt{1-p^2}R\delta a  -R\delta f=0.
\ee
It follows from the analysis below that there exists an infinite dimensional space of pairs $(\delta a,\delta f)$ satisfying \r{r5}. Indeed, for any radial  $\delta a\in C_0^\infty(B(0,1))$, we can solve \r{r5} for $\delta f$, and vice versa. 

Going back to \r{nl.10}, the arguments in the proof of Proposition~\ref{pr_WF} (or its conclusion directly, together with Remark~\ref{rem4}) show that $\delta X(\delta a, \delta f)\in H^s$ and $(\delta a, \delta f)\in H^{s-3/2}$ imply 
\[
\sqrt{1-|x|^2}\delta a-\delta f\in H^{s-1/2}
\]
in the annulus $U := \{0<|x|<1\}$; i.e., the singularities of that particular linear combination in $U$ can be recovered without a derivative loss. Note that for any $x\in U$, the characteristic directions  (zeros of $W$) are given by $\theta=\pm x^\perp/|x|$, and the characteristic codirections --- by $\xi=\pm x/|x|$. Then the integral of $f$, starting from $x$, in a characteristic direction $\theta$ is exactly 
$\sqrt{1-|x|^2|}$. This is the value of $u$ for characteristic directions, see \r{nl.7} and \r{nl.v}, and  confirms \r{nl.19}. 

\subsection{The linearized map for   $a=0$ and $f$ radial has an infinite dimensional kernel} Let now $a$ and $f$ be general radial smooth functions of compact support. Then the characteristic variety of Example~\ref{ex_radial} and Section~\ref{sec_radial} 
\[
\Sigma_0=\{(x,\xi);\; \text{$x$ and $\xi$ are collinear}\}
\]
is included in the characteristic variety $\Sigma$ in this case but the latter can be larger. The Hamiltonian curves (with rays $|x|=R\ge0$) of those examples are still Hamiltonian curves in the present case but there may be more. If $f>0$ in $B(0,1)$, then it is easy to see that the Hamiltonian flow over $B(0,1)$ is the same. 

We study now $\delta X_{a,f}$ for
\be{6}
a=0, \quad \text{$f$ radial}. 
\ee
We also assume that $f$ is smooth and has compact support. With some abuse of notation, we replace $f$ by  $f=f(|x|)$, where $f$ has even smooth extension. By Proposition~\ref{pr1}, 
\be{r7}
\delta X_{0,f} (\delta a, \delta f) = -I_{JBf} \delta a+I_0\delta f,
\ee
see \r{J}. We restrict $\delta X_{0,f} $ to radial $\delta a$, $\delta f$, as well. 

We will use Radon type of parameterization for $I_{JBf} \delta a$ by setting $\omega=\theta^\perp$. Write
\[
 R_{JBf}\delta a  (p,\omega)  =     I_{JBf}    \delta a(p\omega,-\omega^\perp)  = \int \delta(p-\omega\cdot x)Bf(x,\omega^\perp) \delta a(x)\,\d x.
\]
Here $\delta$ is the Dirac Delta function, not to be confused with the variation symbol in $\delta a$, $\delta f$. 
Since $f$ is radial, for any rotation $U$, we have $Bf(Ux,U\omega^\perp) =Bf(x,\omega^\perp) $. Since $\delta a$ is radial as well, we easily get that $I_{JBf} \delta a$ is independent of $\omega$, i.e., $I_{JBf} \delta a= I_{JBf} \delta a(p)$. 
We claim that $I_{JBf} \delta a(p)$ is an even function of $p$. Indeed, set $\omega=(1,0)$. Then 
\begin{alignat*}{2}
R_{JBf} \delta a(-p)  & = \int \delta(-p-x_1)Bf(x,(0,1)) \delta a(x)\,\d x \\
  & = \int \delta(p+x_1)Bf(x,(0,1)) \delta a(x)\,\d x  &\quad& \text{because $\delta$ is even}\\
&= \int \delta(p-x_1)Bf(x,(0,1)) \delta a(x)\,\d x  &&\text{after the change $x_1\mapsto -x_1$}\\
&=R_{JBf} \delta a(p).      
\end{alignat*}
In the last equation, we also used the fact that $f$ is radial. 

To study the kernel of $\delta X_{0,f}$, we write, see \r{r7},
\be{R12}
 R_{JBf\delta} a- R\delta f=0,
\ee
where, with some change of notation again, $R$ is the classical Radon transform acting on radial functions, i.e., considered as a map on functions of a single variable. It is easy to see that, see also \cite{Helgason-Radon},
\[
{R}g(p) = 2\int_\R g\left(\sqrt{p^2+t^2}\right) \d t,  \quad p\ge0.
\]
It is known, see \cite{Natterer-book}, and can be easily seen that this equation can be written in the form
\[
Rg(p) = 2\int_p^\infty \left( 1-\frac{p^2}{r^2} \right)^{-1/2}g(r)\,\d r,  \quad p\ge0.
\]
This an equation of Abel type with explicit inversion given by (see \cite{Gorenflo_Abel, Natterer-book})
\be{R3}
g(r) = -\frac1\pi \int_r^\infty (p^2-r^2)^{-1/2}\frac{\d}{\d p} Rg(p)\, \d p.
\ee
Moreover,  the Abel transform $R$ is given by a composition of the cosine Fourier transform $F_c$ and the zero order Hankel one $H_0$ (see  \cite{Gorenflo_Abel}), with a proper normalization, i.e., $R=F_cH_0$. 
If $h\in C^\infty(\R_+)$ is of compact support, and admits a smooth even extension,  then we get a direct confirmation that the equation $Rg=h$ has a (unique) solution given by $g=H_0F_ch$. Indeed, for such $h$, $F_ch$ has smooth even extension in the Schwartz class, and then $H_0F_ch$ is well defined and solves $\mathcal{R}g=h$. 

This shows that the function $\delta$ in \r{R3} is given by
\be{R4}
\delta f(r) = -\frac1\pi \int_r^\infty (p^2-r^2)^{-1/2}\frac{\d}{\d p} I_{JBf}\delta a (p\omega,\omega^\perp)\, \d p,
\ee
see also \r{r7}. We recall that $I_{JBf}\delta a $ is independent of $\omega$. We summarize this into the following.

\begin{proposition}\label{pr_rad1}
Let $f\in C_0^\infty(\R^2)$ be radial. Then the linearized map $\delta X_{0,f}$ (with $a=0$) has an infinite dimensional kernel, including all radial pairs $(\delta a, \delta f)$ with $\delta a$ smooth function of compact support, and $\delta f$ given by \r{R4}.
\end{proposition}

In other words, besides the inability to recover the singularities (without support restrictions), we actually have an infinite dimensional kernel. Therefore, in this case, the non-trapping condition  is a necessary condition for the problem to be well posed, as well.

\subsection{Non-uniqueness for the Identification Problem for radial $a$, $f$ near $a=0$} 
We show next that not only does the linearized map $\delta X_{a,f}$ can have an infinite dimensional kernel in the case above, but the non-linear map $(a,f)\mapsto X_af$ has a rich set of radial pairs with the same image.
 
\begin{theorem}
Let $a\in C_0^\infty$ and $f\in C_0^\infty$ be radial. Then there exists $f_0\in C_0^\infty$ so that
\be{N1}
X_af= X_0f_0.
\ee
\end{theorem} 

\begin{proof}
We will work again with the Radon transform parameterization $R_af(p,\omega) = X_a(p\omega,-\omega^\perp)$ instead, see \r{2.1R}. As above, it is straightforward to check that
\[
R_af(p,\omega) = R_af(-p,-\omega).
\]
We saw above that  $R_0f(p,\omega)$, denoted there by $Rf$, is actually independent of $\omega$, and  an even function of $p$.  Then  for any $k=0,1,\dots$,
\[
\int R_af(p,\omega)p^k\,\d p = C_k = \text{const.},
\] 
and $C_k=0$ if $k$ is odd. 
Therefore, the integral above is a restriction of the homogeneous polynomial $C_k|\xi|^k$ to the unit sphere. Therefore, $R_af\in \mathcal{S}_H$, and by the Helgason range characterization theorem, see \cite{Helgason-Radon}, 
\r{N1} holds with some $f_0\in \mathcal{S}(\R^2)$. By the support theorem, $f_0$ is compactly supported. 
\end{proof} 

We can actually make this  constructive. By \r{R3}, writing $f_0=f_0(r)$, we get
\[
f_0(r) = -\frac1\pi \int_r^\infty (p^2-r^2)^{-1/2}\frac{\d}{\d p} X_af (p\omega,-\omega^\perp)\, \d p,
\]
recall that $X_af (p\omega,\omega^\perp)$ is independent of $\omega$. 

\appendix
\section{$I_b^*I_a$ as a \PDO}
As explained in Section~\ref{sec_prel}, we view $X_af$ and $I_wf$ as functions on $Z$, with a natural measure $\d z$ there. Then $X_a$, and more generally, $I_w$ have well defined transpose (w.r.t.\ the distribution pairing) and conjugate (w.r.t.\ the $L^2$ product) operators $X_a'$ and $X_a^*$; and $I'_w$, $I^*_w$, respectively.  Below, we use the notation $\theta_\perp$ for the line   given by $s\mapsto p\theta^\perp$.

\begin{proposition}\label{pr_A.1}
\[
I_w'\psi(x) = \int_{S^1} w(x,\theta)\psi(x-(x\cdot\theta)\theta,\theta)\, \d\theta. 
\]
\end{proposition}
\begin{proof}
For $\phi\in C_0^\infty(\R^2)$, $\psi\in C_0^\infty(Z)$, we have
\[
\int_Z (I_w\phi)\psi\, \d z = \int_Z\int_{\R}w(z+s\theta,\theta) \phi(z+s\theta) \psi(z,\theta)\, \d s\,\d z\,\d\theta.
\]
Set $x=z+s\theta$, $z\in\theta_\perp$. For any fixed $\theta$, $(z,s)\mapsto x$ is a diffeomorphism  with a Jacobian equal to $1$. Its  inverse is given by
\[
z=x-(x\cdot\theta)\theta, \quad s=x\cdot\theta. 
\]
Therefore,
\[
\int_Z (I_w\phi)\psi\,\d z= \int_Z\int_{\R^2}w(x,\theta) \phi(x)\psi(x-(x\cdot\theta)\theta,\theta)\,\d x\, \d\theta,
\]
and this proves the proposition. 
\end{proof}

\begin{proposition}\label{pr_symbol} For any two smooth functions $a$ and $b$, 
\[
I_b'I_a f (x)  = \int   \frac{ A\big(x,y,   \frac{x-y}{|x-y|}\big) }{|x-y|}  f(y) \,\d y, 
\]
where
\be{A.A}
A(x,y,\theta) = a(x,\theta) b(y,\theta) + a(x,-\theta) b(y,-\theta).        
\ee
Moreover, $I_b'I_a$ is a classical \PDO\ of order $-1$ with amplitude
\be{A.1amp}
\frac{\pi}{|\xi|}\left( a(x,\xi^\perp/|\xi|)b(y,\xi^\perp/|\xi|) + a(x,-\xi^\perp/|\xi|)\ b(y,-\xi^\perp/|\xi|)\right),
\ee
and principal symbol
\[
\frac{\pi}{|\xi|}\left( a(x,\xi^\perp/|\xi|) b(x,\xi^\perp/|\xi|) + a(x,-\xi^\perp/|\xi|)  b(x,-\xi^\perp/|\xi|)\right).
\]
\end{proposition}

\begin{proof}
By Proposition~\ref{pr_A.1},
\[
\begin{split}
I'_aI_bf(x)&= \int_{S^1} b(x,\theta)\int a( x-(x\cdot\theta)\theta+t\theta,\theta ) f(x-(x\cdot\theta)\theta+t\theta)\,\d t\,\d\theta \\
 &= \int_{S^1} b(x,\theta)\int a(x+t\theta,\theta) f(x +t\theta)\,\d t\,\d\theta.
\end{split}
\]
Split the $t$-integral in two parts: for $t>0$ and for $t<0$, and replace $t$ by $-t$ in the second one to get
\be{A.2}
\begin{split}
I'_aI_bf(x)  &= \int_{S^1} b(x,\theta)\int a(x+t\theta,\theta) f(x +t\theta)\,\d t\,\d\theta\\
  & = \int_{S^1} b(x,\theta)\int_0^\infty a(x+t\theta,\theta) f(x +t\theta)\,\d t\,\d\theta \\
& \quad + \int_{S^1} b(x,\theta)\int_0^\infty a(x-t\theta,\theta) f(x -t\theta)\,\d t\,\d\theta.
\end{split}
\ee
Replace $-\theta$ by $\theta$ in the second integral to get
\be{A.3}
I'_aI_bf(x) = \int_{S^1}\int_0^\infty  \left[ b(x,\theta)a(x+t\theta,\theta)+ b(x,-\theta)a(x+t\theta,-\theta) \right]  f(x +t\theta)\,\d t\,\d\theta.
\ee
Pass to polar coordinates $x=y+t\theta$, centered at $y$ to finish the proof.

To write $I^*_aI_b$ as a \PDO, recall that if the Schwartz kernel of a linear operator is given by $K(x,y,(x-y)/|x-y|)$, then it is a formal \PDO\ with an amplitude given by the Fourier transform of $K$ w.r.t.\ the third variable. Therefore, $I^*_aI_b$ is a formal \PDO\ with amplitude
\[
\begin{split}
\int e^{\i z\cdot\xi} |z|^{-1}A(x,y,z/|z|)\,\d z &= \int_{\R_+\times S^1} e^{\i r\theta\cdot\xi}  A(x,y,\theta)\,\d r\,\d \theta = \pi \int_{S^1}A(x,y,\theta) \delta(\theta\cdot\xi)\,\d\theta\\
&=\frac{\pi}{|\xi|}\left(A(x,y,\xi^\perp/|\xi|)+ A(x,y,-\xi^\perp/|\xi|)\right).
\end{split}
\]
We used here the fact that $A$ is an even function of $\theta$ and that the inverse Fourier transform of $1$ is $\delta$, see also \cite[Theorem~7.1.24]{Hormander1}. 
Since this is a homogeneous function of $\xi$, with an integrable singularity that can be cut-off resulting in a smoothing operator, this completes the proof.
\end{proof}

The mapping properties of those operators are well understood even in the more general setting of the weighted geodesic transform.  
We summarize those properties below. Recall the definition of the Sobolev space $H^s(Z)$ in \r{A.4} first. Given a compact set $K\subset \R^2$, we also use the notation $H^s(K)$ to denote the closed subspace of the distributions in  $H^s(\R^2)$ supported in $K$, see \cite[Chapter~4.5]{Taylor-book2}, where those spaces are denoted by $H_K^s(M)$.

\begin{proposition}\label{pr_I}
For any compact set $K\subset\R^2$, and any $s\ge0$,
\be{A.est1}
I_w : H^{s-1/2}(K) \mapsto H_{\rm comp}^{s}(Z), \quad I'_w :  H_{\rm comp}^{s-1/2}(Z)\mapsto  H^{s}_{\rm loc}(\R^2)  
\ee
are continuous.
\end{proposition}
\begin{proof}
We follow the proof of Proposition~5.1 in \cite{SU-caustics}. We can always assume that $w$ is extended smoothly for $x$ outside $K$ so that it vanishes outside a small neighborhood of $K$.  Then we can replace $\R^2$ and $Z$ by compact manifolds, as explained in Section~\ref{sec_prel}, and work with $f\in C^\infty(\mathbf{T}^2)$. 

Note first that $I_a'I_b: H^s\to H^{s+1}$. Next, if $s\ge0$ is an integer,  for $f$ supported in $K$,
\be{A.5}
\|I_wf\|_{H^{s}(Z)}^2 
\le C\sum_{j\le 2s}\Big| \left(    \partial_p^j   I_wf,I_wf  \right)_{L^2(Z)}\Big| 
=   C\sum_{j\le 2s}\Big| \left( I_w^*   \partial_p^j   I_wf,f  \right)_{L^2(K)}\Big|.
\ee
The term $ I_w^*\partial_p^j I_wf$ is a sum of weighted X-ray transforms of derivatives of $f$ up to order $2s+1$, and is therefore a \PDO\ of order $2s$. This easily implies that for $f\in C^\infty(\mathbf{T}^2)$, 
\[
\|I_wf\|_{H^{s}(Z)}^2 \le C\|f||^2_{H^{s-1/2}(K)}.
\]
The case of general $s\ge0$ follows by interpolation. The estimate then holds for any $f\in H^{s-1/2}(\mathbf{T}^2)$, and therefore, for any $f\in H^{s-1/2}(K)$, as well. 

To prove the second estimate, notice first that $\partial^\alpha I_w^*\psi$ is a sum of operators of the kind $I^*_w$ but with different weights applied to $p$-derivatives of $\psi$ up to order $|\alpha|$. 
Then for any integer $j\ge0$,
\[
|(f,I_a^*\partial_p^j\psi)_{L^2}| =  |(I_a f,\partial_p^j\psi)_{L^2(Z)}|\le C\|f\|_{L^2} \|\psi\|_{H^{j-1/2}}.
\]
This proves the second estimate for $s=0,1,\dots$.  For general $s\ge0$   we use interpolation.
\end{proof} 


\end{document}